\documentclass[twoside,11pt]{article}
\pdfoutput=1
\usepackage{jmlr2e}

\usepackage{mathtools, amssymb, amsthm, xspace, mathabx, verbatim}
\usepackage{caption}
\usepackage{xcolor}
\usepackage{subcaption,graphicx}
\usepackage{graphicx}
\usepackage{enumitem}
\graphicspath{ {images/} }
\RequirePackage[colorlinks,citecolor=blue,urlcolor=blue]{hyperref}
\newtheorem{thm}{Theorem}[]

\newtheorem{lem}[thm]{Lemma}
\newtheorem{cor}[thm]{Corollary}
\newtheorem{defn}[thm]{Definition}
\newtheorem{statement}[thm]{Lemma}
\newcommand{\R}{\ensuremath{\mathbb{R}\xspace}}
\newcommand{\E}{\ensuremath{\mathbb{E}\xspace}}
\newcommand{\set}[1]{\ensuremath{\left\{#1\right\}\xspace}}
\newcommand{\st}{\ensuremath{\;\big|\;\xspace}}
\newcommand{\Lip}{\text{Lip}}
\newcommand{\M}{\mathcal{M}}
\newcommand{\Mput}{\mathcal{M}_\mathrm{put}}

\newcommand{\abs}[1]{\left\lvert#1\right\rvert}
\newcommand{\norm}[1]{\left\lVert#1\right\rVert}
\newcommand{\taubarhat}{\widehat{\widebar{\tau}}}
\newcommand{\tauhat}{\widehat{\tau}}
\let\originalleft\left
\let\originalright\right
\renewcommand{\left}{\mathopen{}\mathclose\bgroup\originalleft}
\renewcommand{\right}{\aftergroup\egroup\originalright}

\jmlrheading{1}{2017}{1-36}{8/17}{TBA}{}{Kitty Mohammed and Hariharan Narayanan}
\ShortHeadings{Manifold Learning Using KDE and PCA}{Mohammed and Narayanan}
\firstpageno{1}

\begin{document}

\title{Manifold Learning Using Kernel Density Estimation and Local Principal Components Analysis}

\author{\name Kitty Mohammed \email kittymohammed1985@gmail.com\\
\addr Department of Statistics\\
University of Washington\\
Seattle, WA 98195-4322 USA
\AND
\name Hariharan Narayanan \email hariharan.narayanan@tifr.res.in\\
\addr School of Technology and Computer Science\\
Tata Institute of Fundamental Research\\
Mumbai, Maharashtra 400 005 India}

\maketitle

\begin{abstract}%
We consider the problem of recovering a $d-$dimensional manifold $\mathcal{M} \subset \R^n$ when provided with noiseless samples from $\mathcal{M}$. There are many algorithms (e.g., Isomap) that are used in practice to fit manifolds and thus reduce the dimensionality of a given data set. Ideally, the estimate $\Mput$ of $\mathcal{M}$ should be an actual manifold of a certain smoothness; furthermore, $\Mput$ should be arbitrarily close to $\M$ in Hausdorff distance given a large enough sample. Generally speaking, existing manifold learning algorithms do not meet these criteria. \citet*{fefferman2016testing} have developed an algorithm whose output is provably a manifold. The key idea is to define an approximate squared-distance function (asdf) to $\mathcal{M}$. Then, $\Mput$ is given by the set of points where the gradient of the asdf is orthogonal to the subspace spanned by the largest $n - d$ eigenvectors of the Hessian of the asdf. As long as the asdf meets certain regularity conditions, $\Mput$ is a manifold that is arbitrarily close in Hausdorff distance to $\M$. In this paper, we define two asdfs that can be calculated from the data and show that they meet the required regularity conditions. The first asdf is based on kernel density estimation, and the second is based on estimation of tangent spaces using local principal components analysis. 
\end{abstract}


\section{Introduction}

It is often the case that high-dimensional data sets have lower-dimensional structure taking the form of a manifold. Manifold learning consists of algorithms that take a high-dimensional data set as input and output a fit of the manifold structure. Many of these algorithms (such as Isomap, Laplacian eigenmaps, locally linear embedding, etc.) are used in practice and have a theoretical literature supporting them. \citet*{ma2011manifold} give a concise overview of these methods.

A drawback of most manifold learning algorithms is that if we are given data from a manifold, their output is not an actual manifold that is close to the original manifold. \citet*{fefferman2016testing} develop an algorithm whose output is provably a manifold of certain smoothness. They start by defining an approximate squared-distance function (asdf) from the data in a manner that uses exhaustive search, utilizing the data only indirectly. Thus, a very large number of potential asdfs are examined before an approximately optimal one is chosen. In this paper, we do away with the exhaustive search, albeit in the specific case of noiseless data that is sampled uniformly from a manifold. \citet{fefferman2016testing} prove a key theorem that states that as long as we are able to define an asdf meeting certain general conditions, their algorithm outputs a set that is a manifold with bounded smoothness and Hausdorff distance to the original manifold. We demonstrate two different methods of estimating the true manifold via asdfs that can be calculated from the data. The two asdfs in our paper are based on 1) kernel density estimation, and 2) approximating the manifold using tangent planes which are in turn approximated with local principal components analysis (PCA).

\cite{ozertem2011locally} learn manifolds by forming a kernel density estimator (KDE) from the data points and finding its $d$-dimensional ridges. We give a more precise definition later, but a ridge is essentially a higher-dimensional analog of the mode and is related to the output set from the algorithm of \citet{fefferman2016testing}. \cite{ozertem2011locally} give a practical method for finding the ridges through a variant of gradient descent where the descent is constrained to the subspace spanned by the largest eigenvectors of the Hessian of the KDE. We state their algorithm in Section \ref{simulations} of our paper and use it to produce simulation results. Although they only apply subspace-constrained gradient descent to find ridges of the KDE, the method is more general and can be used to find ridges of both of our asdfs.

\subsection{Related work}

Manifold learning has existed as an area of statistics and machine learning since the early 2000s. Some classical manifold learning algorithms are Isomap \citep*{tenenbaum2000global}, locally linear embedding \citep*{roweis2000nonlinear}, and Laplacian eigenmaps \citep*{belkin2003laplacian}. Many of these early algorithms rely on spectral graph theory and start off by constructing a graph which is then used to produce a lower-dimensional embedding of the data set. The theoretical guarantees are centered around proving that asymptotically, certain values such as the geodesic distance can be approximated to arbitrary precision. 

More recently, there have been quite a few papers combining ridge estimation with manifold learning (including the work of \citealp*{ozertem2011locally}).  Some early results on ridge estimation are due to \citet*{eberly1996ridges}, \citet*{hall1992ridge}, and \citet*{cheng2004estimating}. Ridge sets can be constructed to estimate a probability density or an embedded submanifold. Theoretical guarantees in this setting have been given by \citet*{genovese2012manifold}, \citet*{genovese2012minimax}, \citet*{genovese2014nonparametric}, and \citet*{chen2015asymptotic}. Of these, the most relevant results for us are from \citet{genovese2014nonparametric}. They prove that as the sample size goes to infinity, their ridge set gets arbitrarily close to an underlying manifold in Hausdorff distance. \cite{fefferman2016testing} also define a procedure related to ridge estimation methods that can be used to estimate an underlying manifold. For our purposes, the major advances of their work are twofold. First, their method is general; as long as a function meets a few conditions, it can be used to define an estimator that can be made arbitrarily close to an underlying manifold in Hausdorff distance. Furthermore, they show that this estimator is itself a manifold with bounded reach (which measures how rough a submanifold can be). Second, their proofs rely on using the implicit function theorem concretely, allowing them to make quantitative statements about the bounds of interest.

\subsection{Outline}

Section \ref{techbkgrd} contains the technical background required to read our main results and proofs. This section starts off with background material on submanifolds, including key definitions, conventions regarding coordinates and projection operators, and important geometric results. Section \ref{sec:assumptions} contains the model assumptions. Section \ref{sec:fmn} summarizes the major theorems we use from \cite{fefferman2016testing}. In Section \ref{ridge}, we summarize the algorithm from \cite{ozertem2011locally} that we use to actually compute the putative manifold. Section \ref{empproc} lists a few key concepts from empirical process theory. We include these because a few of our proofs are simpler when we work in the continuous setting and then argue that a similar result holds for a finite sample from the manifold. Sections \ref{kde} and \ref{localpca} contain the main results of our paper. We provide the precise definition of the asdfs and prove that they do indeed meet the conditions required to apply the theorems contained in Section \ref{sec:fmn}. Finally, we end with simulation results in Section \ref{simulations} and a brief discussion in Section \ref{discussion}.

\section{Technical background and assumptions} \label{techbkgrd}

We now provide the definitions and major theorems that we rely on in the rest of the paper. The results we use the most often are Theorems \ref{federersreach}, \ref{c11fctmanifold}, \ref{thm13fmn}, and \ref{putativemanifold2}; the rest of this section can be referred to as necessary.

\subsection{Manifolds} \label{sec:tech}
This subsection is adapted from \cite{fefferman2016testing}. In the current paper, we use the terms manifold and submanifold interchangeably with compact imbedded $d$-manifold. A closed subset $\mathcal{M} \subset \R^n$ is a \textit{compact imbedded} $d$-\textit{manifold} if the following conditions hold. First, $\mathcal{M}$ is compact. Next, there exists $r_1 > r_2 > 0$ such that for every $z \in \mathcal{M}$ there exists a $d-$dimensional subspace $T_z\mathcal{M}$ of $\R^n$ such that $\mathcal{M} \cap B(z, r_2) = \Gamma \cap B(z, r_2)$ for a patch $\Gamma$ over $T_z\mathcal{M}$ of radius $r_1$, centered at $z$ and tangent to $T_z\mathcal{M}$ at $z$. A \textit{patch} of radius $r$ over $T_z \mathcal{M}$ is a subset $\Gamma := \set{x + \Psi(x)\,|\,x \in B_d(r) \subset T_z \mathcal{M}}$ of $\R^n$ where $\Psi(x): B_d(r) \rightarrow T_z^{\perp} \mathcal{M}$ is a $C^2$-function that is zero at the origin.

The tangent space can be defined in the usual way (corresponding to $T_z\mathcal{M}$) or by using the following definition which applies to arbitrary closed sets $A \subset \R^n$. At a point $a \in A$, $\mathrm{Tan}^0(a, A)$ is the set of vectors $v$ such that for all $\epsilon > 0$, there exists $b \in A$ such that $0 < \abs{a - b} < \epsilon$ and $\abs{v/\abs{v} - \frac{b - a}{\abs{b - a}}} < \epsilon$. Let the \textit{tangent space} $\mathrm{Tan}(a, A)$ be the set of all $x$ such that $x - a \in \mathrm{Tan}^0(a, A)$.

The geometric quantities of a submanifold $\mathcal{M}$ that we are most concerned with are the $d$-dimensional volume $V$ and the \textit{reach} $\tau$. The reach is the largest number such that all points within $\tau$ of $\mathcal{M}$ have a unique closest point on $\mathcal{M}$. Intuitively, the reach governs how ``rough" an embedded submanifold is. For example, the reach of a line with a sharp cusp is zero, and the reach of a linear subspace is infinite. 

The following theorem due to \cite{federer1959curvature} is useful for bounding the distance from a point on a manifold to the tangent space at a nearby point.
\begin{thm} [Federer's reach condition]
	\label{federersreach}
	Let $\mathcal{M}$ be an embedded submanifold of $\R^n$. Then
	\begin{align*}
	\mathrm{reach}(\mathcal{M})^{-1} = \sup \set{2 \|b - a\|^{-2} \|b - \Pi_a b\| \st a, b \in \mathcal{M}, a \neq b}.
	\end{align*}
\end{thm} 

In this paper, we assume regularity conditions on the manifold we draw samples from. We assume it is in $\mathcal{G}$, where $\mathcal{G} = \mathcal{G}(d, V_{max}, \tau_{min})$ is the family of boundaryless $C^2$-submanifolds of the unit ball of $\R^n$ with dimension $d$, volume less than or equal to $V_{max}$, and reach at least $\tau_{min}$. 

Let the \textit{tubular neighborhood} $\mathcal{M}_{\widebar{\tau}}$ be the set of all points within a distance of $\widebar{\tau}$ of $\mathcal{M}$. Now, for points $z \in \mathcal{M}$ and $y \in \mathcal{M}_{\widebar{\tau}}$, denote the projection onto the tangent plane at $z$ by
\begin{align*}
	\Pi_z: \R^n \rightarrow T_z \mathcal{M}.
\end{align*}
A number of our proofs rely on defining the following sets:
\begin{align*}
U_{\widebar{\tau}}^z & := \set{y \st \norm{y - \Pi_z (y)} \leq \widebar{\tau}} \cap \set{y \st \norm{z - \Pi_z (y)} \leq \widebar{\tau}}\\
\widetilde{\mathcal{A}}_{z, \widebar{\tau}} & := U_{\widebar{\tau}}^z \cap \mathcal{M}\\
\mathcal{A}_{z, \widebar{\tau}} & := U_{\widebar{\tau}}^z \cap T_z \mathcal{M}.
\end{align*}
$U_{\widebar{\tau}}^z$ is a cylinder centered at $z$, and $\widetilde{\mathcal{A}}_{z, \widebar{\tau}}$ and $\mathcal{A}_{z, \widebar{\tau}}$ are nearby regions of the manifold and tangent space, respectively. $\mathcal{A}_{z, \widebar{\tau}}$ can also be defined as the projection of the cylinder onto the tangent space; i.e., as $\Pi_z \left(U_{\widebar{\tau}}^z\right)$. These sets are especially useful because, as long as $\widebar{\tau} < \tau$, they allow us to work with a local parametrization of the manifold. As mentioned earlier, manifolds can be defined locally as functions from the tangent space to the normal space. The functions we are working with are in the class $C^{1, 1}$; i.e., they are once continuously differentiable and have a Lipschitz gradient. This is summarized in the next theorem.

\begin{thm}
	\label{c11fctmanifold}
	Let $\mathcal{M} \in \mathcal{G}(d, V_{max}, \tau_{min})$. Let $z \in \mathcal{M}$ and $y \in \mathcal{M}_{\widebar{\tau}}$. When $\widebar{\tau}$ is sufficiently small, there exists a $C^{1, 1}$ function
	\begin{align*}
		F_{z, U_{\widebar{\tau}}^x}: \mathcal{A}_{z, \widebar{\tau}} \rightarrow \Pi_z^{-1}\left(\Pi_z(0)\right)
	\end{align*}
	such that
	\begin{align*}
		\set{y + F_{z, U_{\widebar{\tau}}^z}(y) \st y \in \mathcal{A}_{z, \widebar{\tau}}} = \widetilde{\mathcal{A}}_{z, \widebar{\tau}}.
	\end{align*}
	Additionally, there exists a constant $C$ such that $\mathrm{Lip}(\nabla F_{z, U_{\widebar{\tau}}^z}) \leq C/\tau$.
\end{thm}

The next theorem is from \cite{krantz2012implicit}. It states that $\mathcal{M}$ has positive reach as long as it is embedded in a Euclidean space with strictly higher dimension.
\begin{thm}
	\label{positivereach}
	Let $\mathcal{M}$ be a $d$-dimensional $C^2-$submanifold of $\R^n$. If $n > d$, then $\mathcal{M}$ has positive reach.
\end{thm} 

Now, suppose we want a discrete approximation of a manifold $\mathcal{M}$ at a certain resolution. Let $Y \subset \mathcal{M}$ be an $\eta$-net for $\mathcal{M}$ if for every $p \in \mathcal{M}$ there is a $y \in Y$ such that $\norm{p - y} < \eta$. The following theorem states that the size of an $\eta-$net depends on the geometry of $\M$. 

\begin{thm}
	\label{lemmanifcovnum}
	Let $\mathcal{M} \in \mathcal{G}(d, V_{max}, \tau_{min})$, and let $\mathcal{M}$ be equipped with the Euclidean metric from $\R^n$. For any $\eta > 0$, there exists an $\eta-$net of $\mathcal{M}$ consisting of at most $CV \left(1/\tau^d + 1/\eta^d\right)$ points, where $C$ is a universal constant.
\end{thm}

How well a manifold approximates a point set $\set{x_i}_{i = 1}^N$ can be quantified through the \textit{empirical loss}, which is defined as

\begin{align*}
L_{emp}(\mathcal{M}) := \frac{1}{N} \sum_{i = 1}^{N} d(x_i, \mathcal{M})^2,
\end{align*}
where $d(x_i, \mathcal{M})$ is the length of the projection from $x_i$ onto $\mathcal{M}$.

Given two subsets $X$ and $Y$ of Euclidean space, we can measure the distance between them using the \textit{Hausdorff distance} $H(X, Y)$. This is defined as
\begin{align*}
H(X, Y) := \max \set{\sup_{x \in X} \inf_{y \in Y} \norm{x - y}, \sup_{y \in Y} \inf_{x \in X} \norm{x - y}}.
\end{align*}
It can be shown that, given adequate sampling density, two manifolds that are close in empirical risk to a given point set are also close in Hausdorff distance.\\

\subsection{Model} \label{sec:assumptions}

We assume that we are provided with $\set{y_i}_{1}^{N}$ noiselessly sampled from the uniform distribution on $\mathcal{M} \in \mathcal{G}(d, V_{max}, \tau_{min})$. We take this approach to simplify calculations. The analysis would be similar if the sample came from a (potentially Lipschitz) density bounded away from zero.

\subsection{Approximate squared-distance functions} \label{sec:fmn}
For our purposes the most important results from \cite{fefferman2016testing} are Theorem 13 and Lemma 14. We reproduce them below as Theorems \ref{thm13fmn} and \ref{putativemanifold2}, and give an adapted proof of the latter. It is beyond the scope of this paper to discuss the proof of Theorem \ref{thm13fmn}. We merely note that it relies on the implicit function theorem, so there are concrete bounds on the constants $c_2, \ldots, c_7$ and $C$ that control the geometry of the putative manifold.

Theorem \ref{thm13fmn} states that an \textit{approximate squared-distance function} can be used to recover a manifold with arbitrary precision (with increasing sample size) as long as $F$, a scaled version of the asdf, meets three conditions related to smoothness and curvature. The notation $\partial^\alpha F(x)$ means that given a set of vectors $\alpha := \set{v_1, \dots, v_{\abs{\alpha}}}$, the partial derivative is computed successively in the directions $v_i$. The third condition is the reason for the term asdf: for a small constant $\rho$, $F + \rho^2$ is bounded both above and below by a multiple of $\abs{y}^2 + \rho^2$, the approximate squared distance to the manifold.

Note that the function $F$ always has as its domain the unit ball (or a ball whose radius is not dependent on sample size). $F$ is not the asdf itself, but a related function applied locally after the coordinate system has been scaled up by a constant. This constant is usually a kind of bandwidth parameter that we decrease in order to get a more precise estimate of the manifold. For example, in Section \ref{kde}, we have a scheme to decrease the bandwidth $\sigma$ of the kernel density estimator, and $F$ is the KDE applied to coordinates scaled up by $1/\sigma$. 

The output set from Theorem \ref{thm13fmn} is locally a smooth graph $(x, \Psi(x))$ that lies within a tubular neighborhood of the manifold. Theorem \ref{putativemanifold2} uses bounds on the smoothness of $(x, \Psi(x))$ to show that it lies away from the boundary of the tubular neighborhood, and so it is itself a manifold. We show that it is in fact very close to the original manifold, giving a bound on the Hausdorff distance in terms of a constant that can be made as small as desired. 

\begin{thm}
	\label{thm13fmn}
	Suppose the following conditions hold for a function $F$:
	\begin{enumerate}
		\item $F: B_n(0, 1) \rightarrow \R$ is $C^k$-smooth.
		\item $\partial_{x, y}^\alpha F(x, y) \leq C_0$, where $(x, y) \in B_n(0, 1)$ and $\left|\alpha\right| \leq k$.
		\item For $x \in \R^d$, $y \in \R^{n - d}$ and $(x, y) \in B_n(0, 1)$
		\begin{align*}
			c_1\left(\left|y\right|^2 + \rho^2\right) \leq F(x, y) + \rho^2 \leq C_1\left(\left|y\right|^2 + \rho^2\right),
		\end{align*}
		for $0 < \rho < c$, where $c$ is an arbitrarily small constant depending only on $C_0, c_1, C_1, k,$ and $n$.
	\end{enumerate}
	Then there are constants $c_2, \ldots, c_7$ and $C$  depending only on $C_0, c_1, C_1, k,$ and $n$ such that:
	\begin{enumerate}
		\item For $z \in B_n(0, c_2)$, let $N(z)$ be the subspace of $\R^n$ spanned by the top $n - d$ eigenvectors of $\partial^2 F(z)$. Let $\Pi_{\mathrm{hi}}: \R^n \rightarrow N(z)$ be the orthogonal projection from $\R^n$ to $N(z)$. Then $\left|\partial^{\alpha}\Pi_{\mathrm{hi}}(z)\right| \leq C$ for $z \in B_n(0, c_2)$ and $\left|\alpha\right| \leq k - 2$.
		\item There is a map
		\begin{align*}
			\Psi: B_d(0, c_4) \rightarrow B_{n - d}(0, c_3)
		\end{align*}
		such that $\left|\Psi(0)\right| \leq C \rho$ and $\left|\partial^{\alpha}\Psi\right| \leq C^{\left|\alpha\right|}$ for $1 \leq \left|\alpha\right| \leq k - 2$. The set of all $z = (x, y) \in B_d(0, c_4) \bigtimes B_{n - d}(0, c_3)$ such that
		\begin{align*}
			\set{z \st \Pi_{hi}(z)\partial F(z) = 0} = \set{(x, \Psi(x)) \st x \in B_d(0, c_4)}
		\end{align*}
		is a $C^{k - 2}$-smooth graph.
	\end{enumerate}
\end{thm}

\begin{thm}
\label{putativemanifold2}
	Let $c_1, C_1$, and $C_0$ be the constants appearing in Theorem \ref{thm13fmn}. Assume that $C \rho$ is sufficiently small compared to $r$. Define the putative submanifold
	\begin{align*}
		\mathcal{M}_{\mathrm{put}} = \set{z \in \mathcal{M}_{\min(c_3, c_4)r} \st \Pi_{hi}(z)\partial F(z) = 0}.
	\end{align*}
	Then, $\mathcal{M}_{\mathrm{put}}$ is a submanifold of $\R^n$ which has a reach greater than $cr$, where $c$ depends only on $C_0, c_1, C_1, k, d,$ and $n$. Furthermore, the Hausdorff distance $H\left(\M, \Mput\right)$ is bounded above by $(C^2 + C) \rho$.
\end{thm}

The statement of this theorem assumes that we are provided with the output set from Theorem \ref{thm13fmn}; that is, we are working in the scaled-up coordinates. In the original coordinate system, $\Mput$ is contained in $\mathcal{M}_{\min(c_3, c_4)\sigma r}$, where $\sigma$ is the bandwidth. In this case, the reach is bounded below by $c \sigma r$, and $H\left(\M, \Mput\right)$ is bounded above by $(C^2 + C) \sigma \rho$.

\begin{proof}
$\Mput$ is locally the graph of a $C^{k - 2}$-smooth function $\Psi$. To prove that it is a manifold, it is sufficient to show that it does not intersect the boundary of the tubular neighborhood $\mathcal{M}_{\min(c_3, c_4)r}$. Since Theorem \ref{thm13fmn} gives bounds on $\norm{\partial \Psi}$, we can show by contradiction of the mean value theorem that every point on $\Mput$ is within $\min(c_3, c_4)r/2$ of $\M$.

Suppose there exists a point $\widehat{z}$ on $\Mput$ which is at a distance greater than ${\min(c_3, c_4)r}/{2}$ from $\M$. Let $z := \Pi_{\M} \widehat{z}$. By Theorem \ref{thm13fmn}, there is a point $\widetilde{z} \in \Mput$ such that $\norm{z - \widetilde{z}} < C \rho$. Let $\widetilde{v} \in T_z \M$ be the vector $\Pi_z \left(\widehat{z} - \widetilde{z}\right)$. Let $\widetilde{\Psi}: [0, \norm{\widetilde{v}}] \rightarrow \R^{n - d}$ define a curve on $\Mput$ whose endpoints are $\widehat{z}$ and $\widetilde{z}$. The existence and smoothness of $\widetilde{\Psi}$ are guaranteed by $\Psi$, the $C^{k - 2}$-smooth function that locally defines $\Mput$. The mean value theorem states that there exists a point $x \in [0, \norm{\widetilde{v}}]$ such that
\begin{align*}
\norm{\partial \widetilde{\Psi}(x)} &\geq \frac{1}{\norm{\widetilde{v}}} \norm{\widetilde{\Psi}(z) - \widetilde{\Psi}(z + \widetilde{v})}\\
&= \frac{\norm{\widehat{z} - \widetilde{z}}}{\norm{\widetilde{v}}}\\
& \geq \frac{\norm{z - \widehat{z}} - C \rho}{C \rho}\\
&> \frac{\min(c_3, c_4)r/2 - C \rho}{C \rho}.
\end{align*}
Since $C \rho$ is sufficiently small compared to $\min(c_3, c_4)r/2$,  $\norm{\partial \Psi(x)}$ can be made as large as desired. This contradicts the bound $\norm{\partial \Psi(x)} < C$ and shows that $\Mput$ lies away from the boundary of $\mathcal{M}_{\min(c_3, c_4)r}$. In fact, the expression in the third line above must be less than $C$, which shows that $\norm{z - \widehat{z}} < (C^2 + C) \rho$. Theorem \ref{thm13fmn} states that every point on $\M$ is within $C \rho$ of $\Mput$, so we have the desired bound on the Hausdorff distance.

By Theorem \ref{federersreach}, the reach of $\Mput$ is defined as follows:
\begin{align*}
\mathrm{reach}\left(\Mput\right) = \inf_{\substack{x \neq y\\ x, y \in \Mput}} \frac{\norm{x - y}^2}{2 \norm{y - \Pi_x y}}.
\end{align*}
Let $c'$ be a constant depending on $C_0, c_1, C_1, k, d,$ and $n$. If $\norm{x - y} \geq r/c'$, then 
\begin{align*}
\frac{\norm{x - y}^2}{2 \norm{y - \Pi_x y}} \geq \frac{(r/c')^2}{2 (r/c')}.
\end{align*}
Now, suppose $\norm{x - y} < r/c'$. If $x$ and $y$ are close together, this quantity is controlled by the second derivative of the $C^{k-2}$ function locally defining $\Mput$. That is, $\norm{y - \Pi_x y}$ is on the order of $C^2 \norm{x - y}^2$, implying that 
\begin{align*}
\inf_{\substack{x \neq y\\ x, y \in \Mput\\ \norm{x - y} < r/c'}} \frac{\norm{x - y}^2}{2 \norm{x - \Pi_y x}} \geq \frac{\norm{x - y}^2}{2 c'' C^2 \norm{x - y}^2}
\end{align*}
for some constant $c''$. Therefore,
\begin{align*}
\mathrm{reach}(\Mput) \geq \min\left(\frac{r}{2 c'}, \frac{1}{2 c'' C^2}\right).
\end{align*}
\end{proof}

\subsection{Ridges and gradient descent} \label{ridge}

To actually find a putative manifold using an approximate squared-distance function $F$, we can use a method introduced by \cite{ozertem2011locally}. Let $\partial F$ and $\partial^2 F$ be the gradient and Hessian of $F$, respectively. At a point $x \in \R^n$, let $\set{v_1, \dots, v_n}$ be the eigenvectors of $\partial^2 F$ associated with the eigenvalues $\set{\lambda_1, \dots, \lambda_n}$ (listed in decreasing order). Let $N(z)$ be the subspace of $\R^n$ spanned by the top $n - d$ eigenvectors of $\partial^2 F(z)$. Recall that $\Pi_{\mathrm{hi}}: \R^n \rightarrow N(z)$ is the orthogonal projection from $\R^n$ to $N(z)$. Note that $\Pi_{\mathrm{hi}} = V V^\top$, where $V$ is a matrix whose columns are $[v_1| \dots | v_{n - d}]$.

\cite{ozertem2011locally} give an algorithm to compute the set 
\begin{align*}
\set{z \st \Pi_{hi}(z)\partial F(z) = 0},
\end{align*}
which is termed the $d$\textit{-dimensional ridge} of $F$. This is, of course, the local definition of $\Mput$ from Theorems \ref{thm13fmn} and \ref{putativemanifold2}. In order to find a ridge, an initial set of points is chosen and then iteratively shifted in the direction $VV^\top \partial F$ until a tolerance condition is met. This is essentially a subspace-constrained variant of gradient descent.

\subsection{Empirical processes} \label{empproc}
In Section \ref{kde}, we need to bound various quantities that are functions of the kernel density estimator. This is difficult to do because they are empirical averages over a finite number of samples. It is easier to bound the expectation of these quantities and then bound their difference using results from empirical processes, which we summarize here.

Let $\mathcal{G}$ be a class of functions from $\R^n \rightarrow \R$. If $\mathcal{G}$ consists of bounded functions, the \textit{empirical Rademacher average} is given by 
\begin{align*}
\mathcal{R}_N(\mathcal{G}) = \mathbb{E}_\sigma \frac{1}{N} \left[\sup_{g \in \mathcal{G}}\left(\sum_{i = 1}^{N} \sigma_i g(x_i)\right)\right],
\end{align*}
where $\set{x_i}_1^N$ is an i.i.d. sample from the distribution $\mathcal{P}$ and $\sigma := \set{\sigma_1, \dots, \sigma_N}$ is a vector of Rademacher random variables. (Rademacher random variables take the values $\pm 1$ with equal probability). Letting $\mathcal{P}_N$ denote the empirical distribution on $\set{x_i}_1^N$, the following holds for $0 < \delta < 1$:
\begin{align*}
\mathbb{P} \left[\sup_{g \in \mathcal{G}} \left|\mathbb{E}_{\mathcal{P}_N}g - \mathbb{E}_{\mathcal{P}} g\right| < 2 \mathcal{R}_N (\mathcal{G}) + \sqrt{\frac{2 \log(2/\delta)}{N}}\right] > 1 - \delta.
\end{align*}

It is usually difficult to calculate $\mathcal{R}_N(\mathcal{G})$ directly from the definition. However, the next theorem states an upper bound that is dependent on the size of $\mathcal{G}$, which is often easy to estimate. Let the \textit{covering number} $N(\eta, \mathcal{G}, \|\cdot\|)$ be the minimum number of elements in an $\eta$-net of $\mathcal{G}$ with respect to the norm $\|\cdot\|$. Let the \textit{metric entropy} be defined as $\log N(\eta, \mathcal{G}, \|\cdot\|)$.
The Rademacher complexity can be bounded using a modified form of Dudley's entropy integral \citep*{refineddudley}:

\begin{thm}[Modified Dudley's integral]
\label{dudley}
	\begin{align*}
	\mathcal{R}_N(\mathcal{G}) \leq \inf_{\gamma \geq 0} \left\{4 \gamma + 12 \int_{\gamma}^{\sup_{g \in \mathcal{G}} \|g\|_{\infty}} \sqrt{\frac{\log N(\eta, \mathcal{G}, \|\cdot\|_{\mathcal{L}_2(\mathcal{P}_N)})}{N}} d\eta \right\}.
	\end{align*}
\end{thm}

\section{Kernel density estimation} \label{kde}
Consider the kernel density estimator 
\begin{align*}
	\widetilde{p}_N(x) := \frac{1}{N} \sum\limits_{i = 1}^{N} G_\sigma (x; \, y_i),
\end{align*}
where
$
	G_\sigma (x; \, y)  := C_\sigma e^{-\norm{x - y}^2/2\sigma^2}, 
	C_\sigma  := {{(2 \pi \sigma^2)}^{-d/2}},
$
and $x \in \mathcal{M}_{\sigma}$ (the tubular neighborhood of $\M$ with width $\sigma$). The denominator of $C_\sigma$ has $2 \pi \sigma^2$ raised to the power $d/2$ and not $n/2$ because we are trying to estimate a $d$-dimensional surface. In Theorem \ref{mainkde}, we show that a function based on $\widetilde{p}_N$ can recover a manifold $\mathcal{M}$ when we are given noiseless samples from $\mathcal{M}$.

\subsection{Definition of the asdf}

Recall that Theorem \ref{thm13fmn} must actually be applied in a coordinate system scaled by a bandwidth parameter that becomes more precise with increasing sample size. For the kernel density estimator, this parameter is, of course, $\sigma$. (If we do not scale by $\sigma$, it is clear that $\abs{\partial^\alpha \widetilde{p}_N}$ is bounded above by an increasing function of $1/\sigma$ instead of a universal constant $C_0$). To this purpose, make the following transformations:
\begin{align*}
x &\mapsto x/(\sigma \sqrt{2 \pi})\\
y_i &\mapsto y_i/(\sigma \sqrt{2 \pi})\\
\sigma &\mapsto 1/{\sqrt{2 \pi}}.
\end{align*}
Note that the geometric properties of $\M$ change in the obvious ways: the reach becomes $\tau/(\sigma \sqrt{2 \pi})$ and the volume is $O(V/\sigma^d)$. In the transformed case, let $\widehat{\widebar{\tau}}$, $\widehat{\tau}$, and $\widehat{V}$ denote the analogs of the obvious quantities. For $z$ the projection of $x$ onto $\mathcal{M}$, let $\widetilde{\mathcal{A}}_z := \widetilde{\mathcal{A}}_{z, \widehat{\widebar{\tau}}}$ and ${\mathcal{A}}_z := {\mathcal{A}}_{z, \widehat{\widebar{\tau}}}$. Recall that these are regions of $\mathcal{M}$ and $T_z \mathcal{M}$, respectively, which are near the point $z \in \mathcal{M}$. Define the normalizing factor \begin{align*}
N_f := {\mathrm{Vol} \left(\widetilde{\mathcal{A}}_z \right)} \bigg/ \left({\mathrm{Vol}\left(\mathcal{A}_z\right)} \times \widehat{V}\right).
\end{align*}

The appropriate estimator to analyze is any convenient function of $p_N/N_f$, where
\begin{align*}
	p_N(x) := \frac{1}{N} \sum\limits_{i = 1}^{N} e^{-\pi \norm{x - y}^2}.
\end{align*}
We choose to work with $-\log p_N(x) + \log N_f$ as our potential asdf. The first condition from Theorem \ref{thm13fmn} follows immediately, as seen in the following lemma.

\begin{statement}[]
	\label{lem:firstcondition}
$-\log p_N(x)$ is $C^k$-smooth.
\end{statement}

\begin{proof}
$G_{1/\sqrt{2 \pi}}$ is $C^k$-smooth, so by linearity, $p_N(x)$ is $C^k$-smooth. By the chain rule, $-\log p_N(x)$ is $C^k$-smooth.
\end{proof}

In Lemmas \ref{lem:secondcondition} and \ref{lem:thirdcondition} below, we show that $-\log p_N(x) + \log N_f$ also satisfies the second and third conditions from Theorem \ref{thm13fmn} with high probability. Before detailing the proofs, we briefly discuss our scheme for selecting $\sigma$ and $\widebar{\tau}$. 

\subsection{Selecting the bandwidth \boldmath{$\sigma$}}
The procedure we assume is that a fixed value $\sigma_1$ of $\sigma$ is chosen by the experimenter as well as a value of $\widebar{\tau}$ that depends on $\sigma$. Without making any claims about optimality, we choose $\widebar{\tau} := \sigma^{5/6}$. We prove that there exists a lower bound $K_1$ for $\E p_N(x)$ given $\sigma_1$. We then use empirical processes to show that $p_N(x)$ concentrates around $\E p_N(x)$, and is within $\varepsilon_1$ of $\E p_N(x)$ with high probability. Since $\varepsilon_1$ is a decreasing function of the sample size $N$, we can increase $N$ until $\varepsilon_1 < K_1$ giving us a lower bound for $p_N$. This allows us to derive an upper bound for $\partial^\alpha (-\log p_N(x))$. We also find an expression for $\rho^2$ in terms of $\sigma_1$ and use this to show that condition 3 of Theorem \ref{thm13fmn} holds. If $\rho^2$ is not small enough, we can repeat this procedure using a fixed value $\sigma_{i + 1} := \frac{\sigma_i}{2}$ of $\sigma$ in each subsequent iteration.

\subsection{Bounding \boldmath{$p_N$} in expectation} \label{bounding epn}
To prove the second and third conditions, we need upper and lower bounds for $p_N$. It is more convenient to work initially in the continuous setting, which amounts to bounding $\E p_N$. This is the Gaussian kernel integrated against $\mu_\mathcal{M} := \frac{1}{\widehat{V}} d\mathrm{Vol}(\mathcal{M})$, the measure that is uniform with respect to the volume form. Explicitly,
\begin{align*}
	\E p_N(x) = \int_{\mathcal{M}} e^{-\pi \norm{x - y}^2} d\mu_\mathcal{M}(y).
\end{align*}  
Points on $\mathcal{M}$ that are far away from $x$ do not contribute very much to the value of this integral. In fact, the value of $-\log \E p_N$ is very close to
\begin{align*}
	\widehat{F}_{\widetilde{\mathcal{A}}_z} := -\log \int_{\widetilde{\mathcal{A}}_z} e^{-\pi \norm{x - y}^2} d\mu_\mathcal{M}(y),
\end{align*}
where $z$ is the projection of $x$ onto $\M$.  Define its approximation
\begin{align*}
	\widehat{F}_{\mathcal{A}_z} := -\log \int_{\mathcal{A}_z} e^{-\pi \norm{x - y}^2} N_f d \mathcal{L}_d(y),
\end{align*}
where $\mathcal{L}_d$ is the $d$-dimensional Lebesgue measure on $T_z \mathcal{M}$. 

Decreasing $\sigma$ corresponds to estimating $\mathcal{M}$ with greater precision. Even though this expands the unit ball, leading to $\widehat{\widebar{\tau}} \rightarrow \infty$, the ratio $\widehat{\widebar{\tau}} / \widehat{\tau} \rightarrow 0$. This implies that $\mathcal{A}_z$ and $\widetilde{\mathcal{A}}_z$ are shrinking relatively closer and closer to $z$, and $\widetilde{\mathcal{A}}_z$ is very close to an affine space. Thus, we expect $\widehat{F}_{\widetilde{\mathcal{A}}_z}$ and $\widehat{F}_{\mathcal{A}_z}$ to grow closer together. To prove this, we first need to show that the pushforward of the uniform measure on the manifold has a density $p(y)$ that is close to the uniform density on $B_d\big(\widehat{\widebar{\tau}}\big)$. Of course, we don't want this to be a proper density on $B_d\big(\widehat{\widebar{\tau}}\big)$; we want it to have the same total measure as $\int_{\widetilde{\mathcal{A}}_z} d\mu_{\mathcal{M}}(y')$. In the following lemma, we quantify how much $p(y)$ can deviate from $N_f$ on $\mathcal{A}_z$.

\begin{statement}[]
	\label{lemmapushforward}

	Let $z \in \mathcal{M}$, and let $y \in T_z \mathcal{M}$. The pushforward of $\mu_{\mathcal{M}}$ to $T_z \mathcal{M}$ has density $p(y)$ on $\mathcal{A}_z$ with respect to $\mathcal{L}_d$ such that
	\begin{align*}
		N_f \times \left(1 + \frac{C^2 \widehat{\widebar{\tau}}^2}{\widehat{\tau}^2}\right)^{-d/2} \leq p(y) \leq \left(1 + \frac{C^2 \widehat{\widebar{\tau}}^2}{\widehat{\tau}^2}\right)^{d/2} \times N_f.
	\end{align*}
\end{statement}

\begin{proof}
	Assume that $z$ is the origin and the first $d$ coordinates lie in $T_z \mathcal{M}$. $\mathcal{M}$ is a submanifold of $\R^n$ defined locally by the function $G_{z, U_{\widehat{\widebar{\tau}}}^z}$ that maps $\left(x_1, \dots, x_d\right) \mapsto \left(x_1, \dots, x_d, F_{z, U_{\widehat{\widebar{\tau}}}^z}\right)$. Recall from Theorem \ref{c11fctmanifold} that $F_{z, U_{\widehat{\widebar{\tau}}}^z}: \R^d \rightarrow \R^{n - d}$ is a $C^{1, 1}$ function whose Jacobian $J \in \R^{(n -d) \times d}$ has a Lipschitz constant bounded above by $C/\widehat{\tau}$. $J$ evaluated at $z$ is 0 since $\R^d$ is tangent to $\mathcal{M}$ at $z$; this implies $\norm{J}_F \leq \left(C \widehat{\widebar{\tau}}\right)/\widehat{\tau}$ within a radius of $\widehat{\widebar{\tau}}$, where $\|\cdot\|_F$ is the Frobenius norm. We can find the desired bound on $p(y)$ by finding the ratio of the volume elements of $\mathcal{A}_z$ and $\mathcal{\widetilde{A}}_z$, normalizing this so it integrates to one over $\mathcal{A}_z$, and multiplying by ${\mathrm{Vol} \left(\widetilde{\mathcal{A}}_z \right)} \bigg/ \left(\widehat{V}\right)$. $G_{z, U_{\widehat{\widebar{\tau}}}^z}$ has Jacobian $[I | J^\top]^\top$, allowing us to write
	\begin{align*}
	p(y) = \frac{\displaystyle \mathrm{Vol} \left(\widetilde{\mathcal{A}}_z \right)}{\displaystyle \widehat{V}} \times \frac{\displaystyle \sqrt{\det (I + J^\top J)}}{\displaystyle \int_{\mathcal{A}_z} \sqrt{\det (I + J^\top J)} d\mathcal{L}_d (y)}.
	\end{align*}
	
	Let $\lambda_i$ be the eigenvalues of $J^\top J$. $J^\top J$ is positive semidefinite, so $\lambda_i \geq 0$. Then,
	\begin{align*}
		\sqrt{\det (I + J^\top J)} &= \left(\prod_{i = 1}^{d} \left(1 + \lambda_i\right)\right)^{1/2}\\
		& \leq \left(\prod_{i = 1}^{d} \left(1 + \norm{J^\top J}_F\right)\right)^{1/2}\\
		&\leq \left(1 + \frac{C^2 \widehat{\widebar{\tau}}^2}{\widehat{\tau}^2}\right)^{d/2}.
	\end{align*}
	Since the map from $\mathcal{M}$ to $\R^d$ is a contraction, $1 \leq \sqrt{\det (I + J^\top J)}$. Clearly, we also have
	\begin{align*}
	\mathrm{Vol}\left(\mathcal{A}_z\right) \leq \int_{\mathcal{A}_z} \sqrt{\det (I + J^\top J)} d\mathcal{L}_d (y) \leq \mathrm{Vol}\left(\mathcal{A}_z\right) \left(1 + \frac{C^2 \widehat{\widebar{\tau}}^2}{\widehat{\tau}^2}\right)^{d/2},
	\end{align*}
	which is enough to show the lemma.
\end{proof}

We can use this bound on $p(y)$ to simplify the integration of functions over $\mathcal{M}$. As we mentioned earlier, $- \log \E p_N$ is a function of an integral whose major contribution comes from the region $\widetilde{\mathcal{A}}_z$. (A crude bound suffices for the contribution from the region $\mathcal{M} \setminus \widetilde{\mathcal{A}}_z$). In the following lemma, we show that $\widehat{F}_{\widetilde{\mathcal{A}}_z}$ and $\widehat{F}_{\mathcal{A}_z}$ are very close together. By using the pushforward we can perform both integrals over $\mathcal{A}_z$ using Lebesgue measure.  To do so we need a bound on the ratio of $p(y)$ to $N_f$ (which we have) as well as a bound on the ratio between the integrands. We find that $\widehat{F}_{\widetilde{\mathcal{A}}_z}$ and $\widehat{F}_{\mathcal{A}_z}$ are within a constant $C_f$ of each other. By decreasing $\sigma$, $C_f$ can be made as small as desired.

\begin{statement}[]
\label{claim5}
Let $x \in \mathcal{M}_{\widehat{\widebar{\tau}}}$, and let $z$ be the projection of $x$ onto $\mathcal{M}$. Then, $\abs{\widehat{F}_{\widetilde{\mathcal{A}}_z} - \widehat{F}_{\mathcal{A}_z}} \leq C_f$, where
\begin{align*}
	C_f := \frac{d C^2 \widehat{\widebar{\tau}}^2}{2 \widehat{\tau}^2} + \left(\frac{\widehat{\widebar{\tau}}^4}{ \widehat{\tau}^2 } + \frac{2 \sqrt{2} \widehat{\widebar{\tau}}^3}{\widehat{\tau }}\right)\pi.
\end{align*}
\end{statement}

\begin{proof}
Assume that $z$ is the origin and $T_z \mathcal{M}$ is identified with the first $d$ coordinates. The following chain of inequalities holds, where $y \in T_z \mathcal{M}$, $y' := y + F_{z, U_{\widebar{\tau}}^z}$ is a point on the manifold, and $J$ is the Jacobian of $ F_{z, U_{\widebar{\tau}}^z}$:
\begin{align*}
	\abs{\widehat{F}_{\widetilde{\mathcal{A}}_z} - \widehat{F}_{\mathcal{A}_z}} & = \abs{\log \frac{\displaystyle \int_{\mathcal{A}_z} e^{-\pi \norm{x - y}^2} N_f d \mathcal{L}_d(y)}{\displaystyle \int_{\widetilde{\mathcal{A}}_z} e^{-\pi \norm{x - y'}^2} d\mu_\mathcal{M}(y')}}\\
	& = \abs{\log \frac{\displaystyle \int_{\mathcal{A}_z} e^{-\pi \norm{x - y}^2} N_f d \mathcal{L}_d(y)}{\displaystyle \int_{{\mathcal{A}}_z} e^{-\pi \norm{x - y'}^2} p(y) d \mathcal{L}_d(y)}}\\
	& \leq \abs{\sup_{y \in \mathcal{A}_z} \log \frac{\displaystyle e^{-\pi \norm{x - y}^2} N_f}{\displaystyle e^{-\pi \left\|x - y'\right\|^2} p(y)}}\\
	& \leq \sup_{y \in \mathcal{A}_z} \abs{\log \left(1 + \frac{C^2 \widehat{\widebar{\tau}}^2}{\widehat{\tau}^2}\right)^{d/2}} + \sup_{y \in \mathcal{A}_z} \abs{-(x - y)^2 + \left(x - y'\right)^2} \pi.
\end{align*}

The first term in the last line comes from Lemma \ref{lemmapushforward}. A Taylor expansion (valid for $\abs{x} < 1$) shows that
\begin{align*}
	\log {(1 + x)^{d/2}} = \frac{d x}{2} - \frac{d x^2}{4} + \frac{d x^3}{6} + O(x^4).
\end{align*}
Therefore,
\begin{align*}
	\abs{\log \left(1 + \frac{C^2 \widehat{\widebar{\tau}}^2}{\widehat{\tau}^2}\right)^{d/2}} \leq \frac{d C^2 \widehat{\widebar{\tau}}^2}{2 \widehat{\tau}^2}
\end{align*}
as long as $\widehat{\widebar{\tau}}/\widehat{\tau}$ is smaller than a controlled constant. To bound the other term, we use the law of cosines in conjunction with Theorem \ref{federersreach}, which shows that 
\begin{align*}
\norm{y' - y} = \norm{y' - \Pi_z y'} & \leq \frac{\norm{y' - z}^2}{2 \widehat{\tau}}\\
& \leq \frac{\left(\sqrt{2}\widehat{\widebar{\tau}}\right)^2}{2 \widehat{\tau}}.
\end{align*}
Let $\theta$ be the angle between $y - y'$ and $x - y$. Then, we have:
\begin{align*}
\abs{\norm{y' - x}^2 - \norm{y - x}^2}\pi &= \abs{\norm{y - y'}^2 - 2\norm{y - y'}\norm{y - x} \cos \theta}\pi \\
& \leq \left(\left(\frac{\widehat{\widebar{\tau}}^2}{\widehat{\tau}}\right)^2 + 2\left(\frac{\widehat{\widebar{\tau}}^2}{\widehat{\tau}}\right)(\sqrt{2} \widehat{\widebar{\tau}})\right)\pi.
\end{align*}

Thus, $\abs{\widehat{F}_{\widetilde{\mathcal{A}}_z} - \widehat{F}_{\mathcal{A}_z}} \leq C_f$, where $C_f$ is defined in the statement of the lemma.
\end{proof}

To actually find the lower bound for $p_N$, we bound $\widehat{F}_{\mathcal{A}_z}$ in the next lemma by using a $d$-dimensional Gaussian concentration inequality. The upper bound is much simpler to derive; we include it as well. These bounds are important in verifying the second and third conditions of Theorem \ref{thm13fmn}. The third condition essentially says that our function is an approximate squared-distance function. That is, given a point $x \in \mathcal{M}_{1/\sqrt{2 \pi}}$ and its projection $z \in \mathcal{M}$, we should have upper and lower bounds that are close to $\norm{x - z}^2$. Since our putative asdf is $-\log p_N$, we need bounds for $p_N$ that are within a multiplicative factor of $e^{-\norm{x - z}^2 \pi}$. In the proof of the following lemma we find these pointwise bounds. 

The second condition requires that we find an upper bound for $\partial^{\alpha}(-\log p_N)$. This derivative consists of terms that have powers of $p_N$ in the denominator and combinations of powers of partial derivatives of $p_N$ in the numerator. Thus, we need a uniform lower bound for $p_N$ over $\mathcal{M}_{1/\sqrt{2 \pi}}$; this follows by taking the infimum of the pointwise bound over the tubular neighborhood. We also need bounds for $\abs{\E \partial^{\alpha} p_N}$, but we defer these to the proof of Lemma \ref{lem:secondcondition}.

\begin{statement}[]
\label{claim2}
$p_N(x)$ is bounded in expectation. More precisely, $\inf_{x \in \mathcal{M}_{1/\sqrt{2 \pi}}} \mathbb{E}p_N(x) \geq K_1$, where
\begin{align*}
	K_1 := N_f e^{-{1/2}} \left(1 - 2e^{-{{\left(\widehat{\widebar{\tau}} - \sqrt{d / (2 \pi)}\right)}^2} \pi}\right) e^{-C_f};
\end{align*}
furthermore, $\sup_{x \in \mathcal{M}_{1/\sqrt{2 \pi}}} \mathbb{E}p_N(x) \leq K_2$, where
\begin{align*}
K_2 := e^{C_f} N_f  + e^{-\widehat{\widebar{\tau}}^2 \pi / 2}.
\end{align*}
\end{statement}

\begin{proof}
Let $z$ be the projection of $x$ onto $\mathcal{M}$, and let $y \in T_z \mathcal{M}$.
\begin{align*}
	\mathbb{E}p_N(x) & \geq \int_{\widetilde{\mathcal{A}}_z} e^{-\pi \norm{x - y'}^2} d\mu_\mathcal{M}(y')\\
	& \geq \left(\int_{\mathcal{A}_z}e^{-\pi \norm{x - y}^2} N_f d\mathcal{L}_d(y)\right) e^{-C_f},
\end{align*}
where the second inequality is due to Lemma \ref{claim5}. By orthogonality, $(x - z)^\top (y - z) = 0$, so we rewrite the integral over $\mathcal{A}_z$ as follows: 
\begin{align*}
\int_{\mathcal{A}_z} e^{-\pi \norm{x - y}^2} N_f d\mathcal{L}_d(y) &= N_f e^{-\norm{x - z}^2 \pi} \int_{\mathcal{A}_z} e^{-\pi \norm{z - y}^2} d\mathcal{L}_d(y) \\
& = N_f e^{-\norm{x - z}^2 \pi} \mathbb{P}\left[\norm{y - z} \leq \widehat{\widebar{\tau}}\right],
\end{align*}
where the probability is with respect to a $d-$dimensional multivariate Gaussian with covariance $\frac{1}{2 \pi} I$.
Letting $z$ be the origin for simplicity, we know from standard Gaussian concentration results \citep*{boucheron2013concentration} that 
\begin{align*}
	\mathbb{P}\left[\abs{\norm{y} - \mathbb{E}\norm{y}} \leq t\right] \geq 1 - 2e^{-t^2\pi}
\end{align*}
for $t > 0$. We can calculate $\mathbb{E}\left[\|y\|^2\right]$ and then get a bound for $\E \left[\norm{y}\right]$ by using Jensen's inequality. 

Make the substitutions
\begin{align*}
\set{x_1 \mapsto r \cos{\phi_1}, x_{2 \leq i \leq d - 1} \mapsto r \cos{\phi_i} \prod_{j = 1}^{i - 1}\sin{\phi_j}, x_d \mapsto r \sin{\phi_{d - 1}} \prod_{j = 1}^{d - 2}\sin{\phi_j}},
\end{align*} 
and let
\begin{align*}
dV := r^{d - 1} \prod_{j = 1}^{d - 2} \sin^{d - j - 1} {\phi_j} dr d\phi_1 \dots d\phi_{d - 1}.
\end{align*}
We have
\begin{align*}
\E\left[\norm{y}^2\right] & = \int_{\R^d} \norm{y}^{2} e^{-\norm{y}^2\pi} d {\mathcal{L}_d}(y)\\
& = \int_{0}^{\infty} r^{d + 1} e^{-r^2\pi} dr \times \prod_{j = 1}^{d - 2} \int_{0}^{\pi} \sin^{d - j - 1} \phi_{j}  d\phi_{j} \\ & \quad \times \int_{0}^{2 \pi} d\phi_{d - 1}\\
& = \pi^{-(2 + d)/2} \Gamma\left(1 + d/2\right)\times \prod_{j = 1}^{d - 2} \frac{\sqrt{\pi} \Gamma\left((d - j)/2\right)}{\Gamma\left(1 + (d - j - 1)/2\right)} \times 2 \pi\\
& = \frac{d}{2 \pi}.
\end{align*}
The product in the third line telescopes to $\pi^{(d-2)/2}\bigg/\left(\Gamma(d/2)\right)$; simplifying yields the fourth line.

It follows that $\mathbb{E}\|y\| \leq \sqrt{d / (2 \pi)}$. Setting $t := \widehat{\widebar{\tau}} - \sqrt{d / (2 \pi)}$ (and assuming that $\sigma$ is small enough so $t > 0$), we see that
\begin{align*}
	\mathbb{P}\left[\norm{y} \leq \widehat{\widebar{\tau}}\right] \geq 1 - 2e^{-{{\left(\widehat{\widebar{\tau}} - \sqrt{d / (2 \pi)}\right)}^2} \pi}.
\end{align*}
Consequently,
\begin{align*}
	\mathbb{E}p_N(x) \geq N_f e^{-\norm{x - z}^2 \pi} \left(1 - 2e^{-{{\left(\widehat{\widebar{\tau}} - \sqrt{d / (2 \pi)}\right)}^2} \pi}\right)e^{-C_f}.
\end{align*}
The first part of the lemma follows by taking the infimum over $\mathcal{M}_{1/\sqrt{2 \pi}}$.

To find an upper bound, first write the expectation as
\begin{align*}
\E p_N(x) &= \int_{\mathcal{M}} e^{-\norm{x - y'}^2 \pi} d \mu_{\mathcal{M}}(y')\\
& = \int_{\widetilde{\mathcal{A}}_z} e^{-\norm{x - y'}^2 \pi} d \mu_{\mathcal{M}}(y') + \int_{\mathcal{M} \setminus \widetilde{\mathcal{A}}_z} e^{-\norm{x - y'}^2 \pi} d \mu_{\mathcal{M}}(y').
\end{align*}
The first term can be bounded as follows:
\begin{align*}
\int_{\widetilde{\mathcal{A}}_z} e^{-\norm{x - y'}^2 \pi} d \mu_{\mathcal{M}}(y') & \leq e^{C_f} \int_{{\mathcal{A}}_z} e^{-\norm{x - y}^2 \pi} N_f d \mathcal{L}_d(y)\\
&\leq e^{C_f} N_f e^{-\norm{x - z}^2 \pi} \int_{\mathcal{A}_z} e^{-\norm{z - y}^2 \pi} d \mathcal{L}_d (y)\\
& \leq e^{C_f} N_f e^{-\norm{x - z}^2 \pi}.
\end{align*}
Now consider $y' \in \mathcal{M} \setminus \widetilde{\mathcal{A}}_z$. Since $\norm{z - y'} \leq \norm{x - y'} + \norm{x - z}$ and $\norm{x - z} \leq 1/\sqrt{2 \pi} \leq \widehat{\widebar{\tau}} \leq \norm{z - y'}$, we have
$(\norm{x - z} - \widehat{\widebar{\tau}})^2 \leq \norm{x - y'}^2$. This gives us the following bound for the second term as long as $\widehat{\widebar{\tau}}$ is large enough:
\begin{align*}
\int_{\mathcal{M} \setminus \widetilde{\mathcal{A}}_z} e^{-\norm{x - y'}^2 \pi} d \mu_{\mathcal{M}}(y') & \leq e^{-\left(\norm{x - z} - \widehat{\widebar{\tau}}\right)^2 \pi} \times \int_{\mathcal{M} \setminus \widetilde{\mathcal{A}}_z} d \mu_{\mathcal{M}}(y')\\
& \leq e^{-\widehat{\widebar{\tau}}^2 \pi / 2}.
\end{align*}
Thus, for $x \in \mathcal{M}_{1/\sqrt{2 \pi}}$, 
\begin{align*}
\E p_N(x) & \leq e^{C_f} N_f  + e^{-\widehat{\widebar{\tau}}^2 \pi / 2}.
\end{align*}
\end{proof}

Note that the values we chose for $\sigma$ and $\widebar{\tau}$ are appropriate given our calculations in this section. For decreasing $\sigma$, we would like for $p(y)$ to grow closer to $N_f$ in Lemma \ref{lemmapushforward} and for $C_f$ to tend to zero in Lemma \ref{claim5}; we also need $\left(1 - 2e^{-{{\left(\widehat{\widebar{\tau}} - \sqrt{d / (2 \pi)}\right)}^2} \pi}\right)$, the Gaussian concentration probability, to grow closer to 1 in the previous lemma. Our choice of $\widebar{\tau} := \sigma^{5/6}$ is appropriate given these constraints. For $\sigma$ small enough, $\left(1 - 2e^{-{{\left(\widehat{\widebar{\tau}} - \sqrt{d / (2 \pi)}\right)}^2} \pi}\right) e^{-C_f} \approx 1$ and ${\mathrm{Vol} \left(\widetilde{\mathcal{A}}_z \right)} \bigg/ \left({\mathrm{Vol}\left(\mathcal{A}_z\right)} \times \widehat{V}\right) \approx {1}/{\widehat{V}} \approx {\sigma^d}/{V}$. Thus, $K_1 \approx e^{-1/2} {\sigma^d}/{V}$ and $K_2 \approx {\sigma^d}/{V}$.

\subsection{Finite sample bounds for \boldmath{$p_N$} and \boldmath{$\partial^\alpha (-\log p_N(x))$}}

In Lemma \ref{claim2}, we proved a statement about $\inf_{x \in \mathcal{M}_{1/\sqrt{2 \pi}}} \mathbb{E}p_N(x)$ whereas we really need a statement about $\inf_{x \in \mathcal{M}_{1/\sqrt{2 \pi}}} p_N(x)$. We can use methods from empirical processes to relate these quantities. Let $\mathcal{F}$ consist of functions $f:\mathcal{M} \rightarrow [0, 1]$ where each $f$ has the form $e^{-\pi \norm{x - y'}^2}$ with $y' \in \mathcal{M}$. Here, $x$ is fixed, and each $x \in \mathcal{M}_{1/\sqrt{2 \pi}}$ corresponds to a different $f$. Note that $p_N(x)$ is equivalent to $\E_N f$ and $\E p_N(x)$ is equivalent to $\E f$. We have the tools to prove that for $0 < \delta < 1$,
\begin{align*}
	\mathbb{P} \left[ \sup_{f \in \mathcal{F}} \abs{\mathbb{E}_N f - \mathbb{E}f} \leq \varepsilon_1 \right] \geq 1 - \delta,
\end{align*}
where $\varepsilon_1$ is a function of $\delta$ and $N$. We rewrite the form of the probability bound in part (a) of Lemma \ref{claim3} so that it is in terms of $p_N$ and $\E p_N$. 

In part (b) of Lemma \ref{claim3}, we prove a similar concentration bound for particular derivatives of $e^{\norm{x - z}^2} p_N$. Let $\mathcal{F}_{\beta, v}$ consist of functions $f: \mathcal{M} \rightarrow \R$ where each $f$ is of the form $\partial^\beta_v \left(e^{2 \pi x^\top y'}\right) e^{-\norm{y'}^2\pi}$ with $x \in \mathcal{M}_{1/\sqrt{2 \pi}}$, $y' \in \mathcal{M}$, and $v \in B_n(0, 1)$. These functions are involved in finding an upper bound for $\partial^{\alpha} \left(-\log\left(p_N(x)\right)\right)$.

\begin{statement}[]
\label{claim3}
Let $\mathcal{F}$ be the class of functions consisting of $e^{-\pi \norm{x - y'}^2}$ indexed by $x \in \mathcal{M}_{1/\sqrt{2 \pi}}$. For a given $\beta$ and $v \in B_n(0, 1)$, let $\mathcal{F}_{\beta, v}$ be the class of functions consisting of $\partial^\beta_v \left(e^{2 \pi x^\top y'}\right) e^{-\norm{y'}^2\pi}$ indexed by $x \in \mathcal{M}_{1/\sqrt{2 \pi}}$.

\begin{enumerate}[label=(\alph*)]
\item{
 For $0 < \delta < 1$,
\begin{align*}
	\mathbb{P} \left[ \sup_{x \in \mathcal{M}_{1/\sqrt{2 \pi}}} \abs{p_N(x) - \mathbb{E}p_N(x)} \leq \varepsilon_1 \right] \geq 1 - \delta,
\end{align*}
where
\begin{align*}
	\varepsilon_1 := \frac{24}{\sqrt{N}}\left(\frac{\sqrt{\pi n}}{2} + \sqrt{\log C'}\right) + \sqrt {\frac{2 \log(2/\delta)}{N}}
\end{align*}
and
\begin{align*}
C' := C\widehat{V} 100^d \left(2\sqrt{{2 \pi}/{e}}\right)^n.
\end{align*}
}

\item{
 For $0 < \delta < 1$,
\begin{align*}
	\mathbb{P} \left[ \sup_{x \in \mathcal{M}_{1/\sqrt{2 \pi}}} \abs{\partial^\beta_v \left(e^{2 \pi x^\top y'}\right) e^{-\norm{y'}^2\pi} - \mathbb{E}\left[\partial^\beta_v \left(e^{2 \pi x^\top y'}\right) e^{-\norm{y'}^2\pi}\right]} \leq \varepsilon_{1, \beta} \right] \geq 1 - \delta,
\end{align*}
where
\begin{align*}
	\varepsilon_{1, \beta} := \frac{24}{\sqrt{N}}\left(\frac{\sqrt{\pi n}}{2} + \sqrt{\log C'_{\beta}}\right) + \sqrt {\frac{2 \log(2/\delta)}{N}}
\end{align*}
and
\begin{align*}
C'_{\beta} := C\widehat{V} 100^d \left(2 \left(1 + \sqrt{5 + 4\beta}\right)^{\beta + 1} e^{\left(-1-2\beta + \sqrt{5 + 4 \beta}\right)/4} \left(\frac{2}{\pi}\right)^{-(\beta + 1)/2}\right)^n.
\end{align*}
}
\end{enumerate}
\end{statement}

\begin{proof}
We can bound $\sup_{x} \left|p_N(x) - \mathbb{E}p_N(x)\right|$ through a method from empirical processes by first determining the covering number of $\mathcal{F}$ and then using Dudley's integral. Since $\mathcal{F}$ is a class of Lipschitz functions parametrized by points in $\mathcal{M}_{\widehat{\widebar{\tau}}}$, we can relate its covering number to the covering number of this parameter space.

From empirical process theory, we know that 
\[\mathbb{P} \left[ \sup_{f \in \mathcal{F}} \abs{p_N(x) - \mathbb{E}p_N(x)} \leq 2\mathcal{R}_N(\mathcal{F}) + \sqrt {\frac{2 \log(2/\delta)}{N}}\right] \geq 1 - \delta.\]
$\mathcal{R}_N(\mathcal{F})$ is the Rademacher complexity of $\mathcal{F}$, which can be bounded using Theorem \ref{dudley}. Let $\mathcal{N}(\eta, \mathcal{F}, \norm{\cdot})$ be the covering number at scale $\eta$ with respect to norm $\norm{\cdot}$. Then, 
\begin{align*}
	\mathcal{R}_N(\mathcal{F}) & \leq \inf_{\varepsilon' \geq 0} \left\{4\varepsilon' + 12\int_{\varepsilon'/4}^{\sup_{f \in \mathcal{F}} \sqrt{\hat{\mathbb{E}}[f^2]}} \sqrt{\frac{\log \mathcal{N}\left(\eta, \mathcal{F}, \norm{\cdot}_{\mathcal{L}_2(P_N)}\right)}{N}} d\eta \right\}\\
	& \leq \inf_{\varepsilon' \geq 0} \left\{4\varepsilon' + 12\int_{\varepsilon'/4}^{\sup_{f \in \mathcal{F'}} \sqrt{\hat{\mathbb{E}}[f^2]}} \sqrt{\frac{\log \mathcal{N}\big(\eta, \mathcal{F}, \norm{\cdot}_{\infty}\big)}{N}} d\eta \right\}.
\end{align*}
The second inequality is well-known. Each $f \in \mathcal{F}$ is parametrized by $x \in \mathcal{M}_{1/\sqrt{2 \pi}}$ and is at most $L$-Lipschitz in this parameter. If we can calculate $L$, we can also bound the covering number of $\mathcal{F}$ by relating it to the covering number of the tubular neighborhood.

That is,
\begin{align*}
\mathcal{N}\big(\eta, \mathcal{F}, \norm{\cdot}_\infty \big) &\leq \mathcal{N}\big(\eta/L, \mathcal{M}_{1/\sqrt{2 \pi}}, \norm{\cdot}_2 \big)\\
 &\leq C\widehat{V} 100^d \left(\frac{1}{\eta/L} + 1\right)^n,
\end{align*}
 where the second line follows from taking a 1/100-net of $\mathcal{M}$, placing unit $n$-balls at each net point, and then finding an $\eta/L$-net of those.

Now we find $L$. For simplicity, assume $x$ has coordinates $\left(x_1, \dots, x_n\right)$ that have been centered around any point on the manifold. By the symmetry of $\norm{x}$, we only need to consider one coordinate.
\begin{align*}
\Lip(f) = \sup_{f \in \mathcal{F}} \norm{\nabla f}
& \leq \sup_{x \in \mathcal{M}} \norm{\nabla \left(e^{-\norm{x}^2 \pi}\right)}\\
& \leq \sup_{x \in \R^n} \abs{\frac{\partial}{\partial x_1}e^{-\norm{x}^2 \pi}}\\
& = 2 e^{-x_1^2 \pi} x_1 \bigg|_{x_1 = {1}/{\sqrt{2 \pi}}}\\
& = \sqrt{\frac{2 \pi}{e}}\\
& =: L.
\end{align*}
Since $\eta$ ranges between 0 and 1, $L/\eta > 1$. Define $C' := C\widehat{V} 100^d (2\sqrt{{2 \pi}/{e}})^n$; then,
 \begin{align*}
 \mathcal{N}\big(\eta, \mathcal{F}, \norm{\cdot}_\infty \big) \leq C'\eta^{-n}.
 \end{align*}

Using the monotonicity of $\log$ and the square root,
\begin{align*}
\mathcal{R}_N(\mathcal{F}) & \leq 12\int_{0}^{1} \left(\sqrt{\frac{\log C'}{N}} + \sqrt{\frac{-n \log \eta}{N}}\right) d\eta\\
&= \frac{12}{\sqrt{N}}\left(\frac{\sqrt{\pi n}}{2} + \sqrt{\log C'}\right).
\end{align*}

Thus, with high probability,
\begin{align*}
\sup_{f \in \mathcal{F}} |p_N(x) - \mathbb{E}p_N(x)| & \leq \frac{24}{\sqrt{N}}\left(\frac{\sqrt{\pi n}}{2} + \sqrt{\log C'}\right) 
+ \sqrt {\frac{2 \log(2/\delta)}{N}},
\end{align*}
which proves (a).

The proof of part (b) is nearly the same, with the only difference being in the covering number of the parameter space. If each $f \in \mathcal{F}_{\beta, v}$ is at most $L_{\beta}$-Lipschitz, then
\begin{align*}
\mathcal{N}\big(\eta, \mathcal{F}_{\beta}, \norm{\cdot}_\infty \big) \leq C'_{\beta}\eta^{-n},
\end{align*}
where $C'_{\beta} := C\widehat{V} 100^d (2 L_{\beta})^n$. Since
\begin{align*}
\partial^\beta_v \left(e^{2 \pi x^\top y'}\right) e^{-\norm{y'}^2\pi} = e^{2 \pi x^\top y'} (2 \pi)^{\beta} \left(y'^\top v\right)^{\beta} e^{-\norm{y'}^2\pi},
\end{align*}
we have
\begin{align*}
\Lip(f) & = \sup_{f \in \mathcal{F}_{\beta, v}} \norm{\nabla f}\\
& \leq \sup_{x \in \mathcal{M}_{1/\sqrt{2 \pi}}, y' \in \R^n, v \in B_n(0, 1)} \norm{\nabla \left(e^{2 \pi x^\top y'}\right) (2 \pi)^{\beta} \left(y'^\top v\right)^{\beta} e^{-\norm{y'}^2\pi}}\\
& = \sup_{x \in \mathcal{M}_{1/\sqrt{2 \pi}}, y' \in \R^n, v \in B_n(0, 1)} (2 \pi)^{\beta} \abs{y'^\top v}^{\beta} e^{-\norm{y'}^2\pi} \sqrt{\sum_{i = 1}^{n} \left(2 \pi y'_i e^{2\pi x^\top y'} \right)^2}\\
& = \sup_{x \in \mathcal{M}_{1/\sqrt{2 \pi}}, y' \in \R^n, v \in B_n(0, 1)} (2 \pi)^{\beta} \abs{y'^\top v}^{\beta} e^{-\norm{y'}^2\pi} 2 \pi e^{2\pi x^\top y'} \norm{y'}\\
& \leq \sup_{y' \in \R^n} (2 \pi)^{\beta + 1} \norm{y'}^{\beta + 1} e^{-\norm{y'}^2\pi + \sqrt{2 \pi} \norm{y'}}.
\end{align*} 
In the third line, $\set{y'_1, \dots, y'_n}$ are the components of $y'$. The final line follows by the Cauchy-Schwarz inequality, which shows that $\abs{y'^\top v} \leq \norm{y'} \norm{v} \leq \norm{y'}$ and $2 \pi x^\top y' \leq 2\pi \norm{x} \norm{y'} \leq \sqrt{2 \pi} \norm{y'}$. Differentiating with respect to $\norm{y'}$ and setting equal to zero shows that the supremum is achieved at $\norm{y'} = (1 + \sqrt{5 + 4 \beta})/(2 \sqrt{2 \pi})$. We can substitute this back in to set
\begin{align*}
L_{\beta} := \left(1 + \sqrt{5 + 4\beta}\right)^{\beta + 1} e^{\left(-1-2\beta + \sqrt{5 + 4 \beta}\right)/4} \left(\frac{2}{\pi}\right)^{-(\beta + 1)/2}.
\end{align*}
\end{proof}

It follows directly that $K_1 - \varepsilon_1 \leq p_N \leq K_2 + \varepsilon_1$ with high probability. For large enough $N$, $K_1/2 \leq p_N \leq 2K_2$. In the next lemma, we prove that a corresponding result holds for $\partial^\alpha (-\log p_N(x))$ (which is exactly the second condition from Theorem \ref{thm13fmn}). The derivation is more technical but the intuition is based on the arguments in Section \ref{bounding epn}.

\begin{statement}[]
	\label{lem:secondcondition}
$\partial^\alpha (-\log p_N(x)) \leq C_0$ for $x \in B_n(0,1/\sqrt{2 \pi})$, $|\alpha| \leq k$, and $C_0$ depending only on $n$ and $k$.
\end{statement}

\begin{proof}
Start by defining
\begin{align*}
q_N(x) = \frac{1}{N} \sum\limits_{i = 1}^{N}   e^{-\norm{z - y_i}^2\pi} e^{-2(x - z)^\top (z - y_i)\pi},
\end{align*}
where $z$ is the projection of $x$ onto $\mathcal{M}$. Then,
\begin{align*}
\partial^\alpha (-\log p_N(x)) & = \partial^\alpha \left(-\log \left(e^{-\norm{x - z}^2 \pi} q_N(x)\right)\right).
\end{align*}
A result due to \cite{nemirovski2004interior} shows that
\begin{align*}
\sup_{|v_i| \leq 1} \abs{\partial_{v_1} \dots \partial_{v_{\abs{\alpha}}} F(x)} \leq \sup_{|v| \leq 1} \abs{\partial_{v}^{\abs{\alpha}} F(x)}
\end{align*}
for $C^k$-smooth $F$, implying that we do not need to bound mixed partials. 

It is straightforward to calculate $C_{0, 1} := \sup \abs{\partial^\alpha \left(-\log \left(e^{-\norm{x - z}^2 \pi}\right)\right)}$. (The supremum is also over $\abs{\alpha} \leq k$). To get an upper bound for $\partial^\alpha \left(-\log q_N(x)\right)$, we first write it as an expression involving powers of $q_N(x)$ and partials of $q_N(x)$. For example, if $\alpha = \set{x_1, x_1, x_1, x_1}$,
\begin{align*}
\frac{\partial^4 (-\log q_N(x))}{\partial x_1^4} &= 6 \frac{\left(\frac{\partial q_N}{\partial x_1}\right)^4}{q_N(x)^4} - 12 \frac{\left(\frac{\partial q_N}{\partial x_1}\right)^2 \left(\frac{\partial^2 q_N}{\partial x_1^2}\right)}{q_N(x)^3} + 3 \frac{\left(\frac{\partial^2 q_N}{\partial x_1^2}\right)^2}{q_N(x)^2}\\
 &+ 4 \frac{\left(\frac{\partial q_N}{\partial x_1}\right) \left(\frac{\partial^3 q_N}{\partial x_1^3}\right)}{q_N(x)^2} - \frac{\left(\frac{\partial^4 q_N}{\partial x_1^4}\right)}{q_N(x)}.
\end{align*}
Faa di Bruno's formula is an explicit representation of this expression; the number of terms and the coefficients depend on $\abs{\alpha}$. We can find a suitable $C_0$ if we can calculate a lower bound for $q_N$ and an upper bound for $\abs{{\partial^\beta_v q_N}}$ where $\beta \leq \abs{\alpha}$. The first bound follows from two previous lemmas. Lemma \ref{claim2} shows that $\E p_N(x) \geq K_1$, and Lemma \ref{claim3} shows that $p_N$ is within $\varepsilon_1$ of its expectation with high probability. Since $q_N(x) \geq p_N(x)$ and $\varepsilon_1$ can be made smaller than $K_1/2$, $q_N \geq K_1/2$ with high probability for $N$ sufficiently large. $K_1$ is a function of $d$, and $\varepsilon_1$ is a function of $n$.

To bound the partials of $q_N$, we start off by using the second part of Lemma \ref{claim3}, which shows
\begin{align*}
\abs{{\partial^\beta_v q_N}}  \leq \abs{\E {\partial^\beta_v q_N}} + \varepsilon_{1, \beta}.
\end{align*}
Let $z$ be the origin and let the first $d$ coordinates lie in $T_z \mathcal{M}$. We can write the expectation as
\begin{align*}
\abs{\E {\partial^\beta_v q_N}} & = \abs{\int_{\mathcal{M}} \partial^\beta_v \left(e^{2 \pi x^\top y'}\right) e^{-\norm{y'}^2\pi} d \mu_{\mathcal{M}}(y')}\\
& \leq \abs{\int_{\widetilde{\mathcal{A}}_z} e^{2 \pi x^\top y'} (2 \pi)^{\beta} \left(y'^\top v\right)^{\beta} e^{-\norm{y'}^2\pi} d \mu_{\mathcal{M}}(y')}\\ &+ \abs{\int_{\mathcal{M} \setminus \widetilde{\mathcal{A}}_z} e^{2 \pi x^\top y'} (2 \pi)^{\beta} \left(y'^\top v\right)^{\beta} e^{-\norm{y'}^2\pi} d \mu_{\mathcal{M}}(y')}.
\end{align*}

We first bound the integral over $\mathcal{M} \setminus \widetilde{\mathcal{A}}_z$. For $\widehat{\widebar{\tau}}$ large enough, the local extrema of $\partial^\beta_v \left(e^{2 \pi x^\top y'}\right) e^{-\norm{y'}^2\pi}$ with respect to $y'$ lie within $U_{\widehat{\widebar{\tau}}}^{z}$.  Since 
\begin{align*}
\lim_{\norm{y'} \rightarrow \infty} \partial^\beta_v \left(e^{2 \pi x^\top y'}\right) e^{-\norm{y'}^2\pi} = 0,
\end{align*}
$\abs{\partial^\beta_v \left(e^{2 \pi x^\top y'}\right) e^{-\norm{y'}^2\pi}}$ is decreasing with increasing $\norm{y'}$ in $\mathcal{M} \setminus \widetilde{\mathcal{A}}_z$. The following holds, where $y_0 \in \partial \widetilde{\mathcal{A}}_z$:
\begin{align*}
& \abs{\int_{\mathcal{M} \setminus \widetilde{\mathcal{A}}_z} e^{2 \pi x^\top y'} (2 \pi)^{\beta} \left(y'^\top v\right)^{\beta} e^{-\norm{y'}^2\pi} d \mu_{\mathcal{M}}(y')}\\
\leq &  e^{2 \pi x^\top y_0} (2 \pi)^{\beta} \left(y_0^\top v\right)^{\beta} e^{-\norm{y_0}^2\pi} \int_{\mathcal{M} \setminus \widetilde{\mathcal{A}}_z} d \mu_{\mathcal{M}}(y')\\
\leq & e^{2 \pi \norm{x} \norm{y_0}} (2 \pi)^{\beta} \norm{y_0}^{\beta} \norm{v}^{\beta} e^{-\norm{y_0}^2\pi}\\
\leq & (2 \sqrt{2} \pi \widehat{\widebar{\tau}})^{\beta} e^{-2 \pi \widehat{\widebar{\tau}}^2 + 2 \sqrt{\pi}\widehat{\widebar{\tau}} }.
\end{align*}

The integral over $\widetilde{\mathcal{A}}_z$ can be bounded by relating it to the corresponding integral over $\mathcal{A}_z$. Let $y$ be the projection of $y' \in \widetilde{\mathcal{A}}_z$ onto $\mathcal{A}_z$. Then,
\begin{align*}
& \abs{\int_{\widetilde{\mathcal{A}}_z} e^{2 \pi x^\top y'} (2 \pi)^{\beta} \left(y'^\top v\right)^{\beta} e^{-\norm{y'}^2\pi} d \mu_{\mathcal{M}}(y')}\\
= & \abs{\int_{{\mathcal{A}}_z} e^{2 \pi x^\top y'} (2 \pi)^{\beta} \left(y'^\top v\right)^{\beta} e^{-\norm{y'}^2\pi} p(y) d {\mathcal{L}_d}(y)}\\
\leq & \abs{\int_{{\mathcal{A}}_z} e^{2 \pi x^\top y} \left(y'^\top v\right)^{\beta} e^{-\norm{y}^2\pi} d {\mathcal{L}_d}(y)} \times N_f \left(1 + \frac{C^2 \widehat{\widebar{\tau}}^2}{\widehat{\tau}^2}\right)^{d/2}\\
& \quad \times \sup_{x, y'} \abs{{e^{2 \pi x^\top\left(y' - y\right)}}} \times \sup_{y'} \abs{e^{\left(\norm{y'}^2 - \norm{y}^2\right) \pi}} \times (2 \pi)^{\beta} \\
\leq & \abs{\int_{{\mathcal{A}}_z} \left(y'^\top v\right)^{\beta} e^{-\norm{y}^2\pi} d {\mathcal{L}_d}(y)} \times N_f \left(1 + \frac{C^2 \widehat{\widebar{\tau}}^2}{\widehat{\tau}^2}\right)^{d/2}\\
& \quad \times e^{2 \sqrt{2 \pi} \taubarhat^2/\left(2 \tauhat\right)} \times e^{\pi \taubarhat^4/\tauhat^2} \times (2 \pi)^{\beta}.
\end{align*}
The third line comes from relating $p(y)$ and $N_f$ (Lemma \ref{lemmapushforward}) and bounding the change in the integrand due to projecting $y'$ onto $\mathcal{A}_z$. We do not project $\left(y'^\top v\right)^{\beta}$ because it can equal zero. The fourth line follows by noting that $x^\top y = 0$ by orthogonality and that $\norm{y' - y} \leq {\taubarhat^2}/{\tauhat}$ by Federer's reach condition. The reach condition also shows that $\left(y'^{\top} v\right)^{\beta}$ is a polynomial whose terms either lie in $\R^d$ or have arbitrarily small coefficients. This can be used to bound the integral. Starting off by applying the triangle inequality for integrals and then the Cauchy-Schwarz inequality, we have
\begin{align*}
\abs{\int_{{\mathcal{A}}_z} \left(y'^\top v\right)^{\beta} e^{-\norm{y}^2\pi} d {\mathcal{L}_d}(y)} & \leq \int_{{\mathcal{A}}_z} \norm{y'}^{\beta} \norm{v}^{\beta} e^{-\norm{y}^2\pi} d {\mathcal{L}_d}(y)\\
& \leq \int_{{\mathcal{A}}_z} \left(\norm{y} + \norm{y' - y}\right)^{\beta} e^{-\norm{y}^2\pi} d {\mathcal{L}_d}(y)\\
& \leq \int_{{\mathcal{A}}_z} \norm{y}^{\beta} e^{-\norm{y}^2\pi} d {\mathcal{L}_d}(y) \\ & \quad + \sum_{i = 1}^{\beta} \binom{\beta}{i} \left(\frac{\taubarhat^2}{\tauhat}\right)^{i}\int_{{\mathcal{A}}_z} \norm{y}^{\beta - i} e^{-\norm{y}^2\pi} d {\mathcal{L}_d}(y) \\
& \leq 2 \int_{{\mathcal{A}}_z} \norm{y}^{\beta} e^{-\norm{y}^2\pi} d {\mathcal{L}_d}(y).
\end{align*}
The second line holds because $\norm{v} \leq 1$ and $\norm{y'}$ can be bounded using the triangle inequality. The third line follows after expanding $\left(\norm{y} + \norm{y' - y}\right)^{\beta}$, substituting in the bound for $\norm{y' - y}$, and rearranging. Each term in the summation can be made arbitrarily small, which gives the fourth line. This integral is a function of the moments (of order $\beta$ or less) of a $d$-dimensional Gaussian with covariance $1/\left(2 \pi\right) I$. We can calculate it using spherical coordinates, following the calculation of $\E\left[\norm{y}^2\right]$ in Lemma \ref{claim2}. We have
\begin{align*}
2 \int_{{\mathcal{A}}_z} \norm{y}^{\beta} e^{-\norm{y}^2\pi} d {\mathcal{L}_d}(y) &\leq 2 \int_{\R^d} \norm{y}^{\beta} e^{-\norm{y}^2\pi} d {\mathcal{L}_d}(y)\\
& = 2 \int_{0}^{\infty} r^{\beta + d - 1} e^{-r^2\pi} dr \times \prod_{j = 1}^{d - 2} \int_{0}^{\pi} \sin^{d - j - 1} \phi_{j}  d\phi_{j} \\ & \quad \times \int_{0}^{2 \pi} d\phi_{d - 1}\\
& = \pi^{-(\beta + d)/2} \Gamma\left(\frac{\beta + d}{2}\right)\times \prod_{j = 1}^{d - 2} \frac{\sqrt{\pi} \Gamma\left((d - j)/2\right)}{\Gamma\left(1 + (d - j - 1)/2\right)} \times 2 \pi\\
& = 2 \pi^{-\beta/2} \frac{\Gamma\left((\beta + d)/2\right)}{\Gamma(d/2)}\\
& =: C_{\beta}.
\end{align*}

Therefore, for large enough $\taubarhat$,
\begin{align*}
\abs{\E {\partial^\beta_v q_N}} & \leq 2 (2 \pi)^{\beta} N_f C_{\beta} + (2 \sqrt{2} \pi \widehat{\widebar{\tau}})^{\beta} e^{-2 \pi \widehat{\widebar{\tau}}^2 + 2 \sqrt{\pi}\widehat{\widebar{\tau}} }\\
& \leq 3 (2 \pi)^{\beta} N_f C_{\beta}.
\end{align*}
If $N$ is large enough, $\varepsilon_{1, \beta}$ will be smaller than $(2 \pi)^{\beta} N_f C_{\beta}$ with high probability, implying that $\abs{{\partial^\beta_v q_N}} \leq 4(2 \pi)^{\beta} N_f C_{\beta}$. Let $\beta' \leq \abs{\alpha}$ be the value of $\beta$ for which this is maximized. Each term of $\abs{\partial^\alpha \left(-\log q_N(x)\right)}$ is bounded above in absolute value by a multiple of ${4(2 \pi)^{\beta'} N_f C_{\beta'}}/{\left(K_1/2\right)}$ raised to a power less than or equal to $\abs{\alpha}$. Therefore, using the triangle inequality and letting $\abs{\alpha} \leq k$, $\abs{\partial^\alpha \left(-\log q_N(x)\right)} \leq C_{0, 2}$, a constant. The factors of $N_f$ in $K_1$ and $\abs{{\partial^\beta_v q_N}}$ cancel each other out, so $C_{0, 2}$ is a function of $n$ and $k$. Setting $C_0 := C_{0, 1} + C_{0, 2}$ yields the lemma.
\end{proof}

\subsection{\boldmath{$-\log p_N(x) + \log N_f$} is an asdf}
In the next two lemmas, we prove that the third condition of Theorem \ref{thm13fmn} holds. Recall from the proof of Lemma \ref{claim2} that $\E p_N$ can be bounded above and below to within a multiplicative factor of $N_f e^{-\norm{x - z}^2 \pi}$. By taking logarithms and defining a suitable constant $\widetilde{\rho}$, we show in Lemma \ref{claim6} that the third condition holds for $-\log \E p_N + \log N_f$. In Lemma \ref{lem:thirdcondition}, we show that this condition also holds for $-\log p_N + \log N_f$ as long as we modify $\widetilde{\rho}$ to take into account the concentration bound from Lemma \ref{claim3}.

\begin{statement}[]
\label{claim6}
For $x \in \mathcal{M}_{1/\sqrt{2 \pi}}$ and $z$ the projection of $x$ onto $\mathcal{M}$, 
\begin{align*}
c_1 \left(\norm{x - z}^2 \pi + \widetilde{\rho}^2 \right) \leq -\log \E p_N + \log N_f + \widetilde{\rho}^2 \leq C_1 \left(\norm{x - z}^2 \pi + \widetilde{\rho}^2 \right)
\end{align*}
for $0 < \widetilde{\rho} < c$, with $c$ depending on $C_0, c_1, C_1, k, n$.
\end{statement}

\begin{proof}
Let $x \in \mathcal{M}_{1/\sqrt{2 \pi}}$ and let $z$ be its projection onto $\mathcal{M}$. Then we can bound $\E p_N(x)$ by calculating the expectation separately over $\widetilde{\mathcal{A}}_z$ and $\mathcal{M} \setminus \widetilde{\mathcal{A}}_z$:
\begin{align*}
\E p_N(x) &= \int_{\mathcal{M}} e^{-\norm{x - y'}^2 \pi} d \mu_{\mathcal{M}}(y')\\
& = \int_{\widetilde{\mathcal{A}}_z} e^{-\norm{x - y'}^2 \pi} d \mu_{\mathcal{M}}(y') + \int_{\mathcal{M} \setminus \widetilde{\mathcal{A}}_z} e^{-\norm{x - y'}^2 \pi} d \mu_{\mathcal{M}}(y').
\end{align*}
The first term can be bounded as follows:
\begin{align*}
\int_{\widetilde{\mathcal{A}}_z} e^{-\norm{x - y'}^2 \pi} d \mu_{\mathcal{M}}(y') & \leq e^{C_f} \int_{{\mathcal{A}}_z} e^{-\norm{x - y}^2 \pi} N_f d \mathcal{L}_d(y)\\
&\leq e^{C_f} N_f e^{-\norm{x - z}^2 \pi} \int_{\mathcal{A}_z} e^{-\norm{z - y}^2 \pi} d \mathcal{L}_d (y)\\
& \leq e^{C_f} N_f e^{-\norm{x - z}^2 \pi}.
\end{align*}
Now consider $y' \in \mathcal{M} \setminus \widetilde{\mathcal{A}}_z$. Since $\norm{z - y'} \leq \norm{x - y'} + \norm{x - z}$ and $\norm{x - z} \leq 1/\sqrt{2 \pi} \leq \widehat{\widebar{\tau}} \leq \norm{z - y'}$, we have
$(\norm{x - z} - \widehat{\widebar{\tau}})^2 \leq \norm{x - y'}^2$. This gives us the following bound for the second term:
\begin{align*}
\int_{\mathcal{M} \setminus \widetilde{\mathcal{A}}_z} e^{-\norm{x - y'}^2 \pi} d \mu_{\mathcal{M}}(y') \leq e^{-\left(\norm{x - z} - \widehat{\widebar{\tau}}\right)^2 \pi}.
\end{align*}
Thus, letting
\begin{align*}
C_{neg} := \frac{\displaystyle e^{-\widehat{\widebar{\tau}}^2 \pi + 2\widehat{\widebar{\tau}}\pi\norm{x - z} } } {\displaystyle N_f e^{C_f} },
\end{align*}
we have
\begin{align*}
\E p_N(x) & \leq e^{C_f} N_f e^{-\norm{x - z}^2 \pi} \left(1 + C_{neg}\right) \\
& \leq e^{C_f} N_f e^{-\norm{x - z}^2 \pi}  \left(\frac{1}{1 - C_{neg}}\right).
\end{align*}
The second inequality holds because $C_{neg}$ is arbitrarily small for small $\sigma$, so we can use the Taylor expansion
\begin{align*}
\frac{1}{1 - x} = \sum_{i = 0}^{\infty} x^i.
\end{align*}

Next, we find a lower bound for $\E p_N(x)$ as in Lemma \ref{claim2}.
\begin{align*}
\E p_N(x) &\geq \left(1 - C_{neg}\right) \int_{\widetilde{\mathcal{A}}_z} e^{-\norm{x - y'}^2 \pi} d \mu_{\mathcal{M}}(y')\\
 & \geq \left(1 - C_{neg}\right) e^{-C_f} \int_{{\mathcal{A}}_z} e^{-\norm{x - y}^2 \pi} N_f d \mathcal{L}_d(y)\\
&\geq \left(1 - C_{neg}\right) e^{-C_f} N_f e^{-\norm{x - z}^2 \pi} \int_{\mathcal{A}_z} e^{-\norm{z - y}^2 \pi} d \mathcal{L}_d (y)\\
& \geq \left(1 - C_{neg}\right) e^{-C_f} N_f e^{-\norm{x - z}^2 \pi} \left(1 - 2e^{-{{\left(\widehat{\widebar{\tau}} - \sqrt{d / (2 \pi)}\right)}^2}\pi}\right).
\end{align*}

Now, let $\alpha := -\log(1 - C_{neg})$ and $\beta := -\log\left(1 - 2e^{-{{\left(\widehat{\widebar{\tau}} - \sqrt{d / (2 \pi)}\right)}^2}\pi}\right)$. We have shown the following:
\begin{align*}
-C_f + \norm{x-z}^2 \pi - \alpha \leq -\log \E p_N(x) +\log N_f \leq \alpha + C_f + \norm{x - z}^2 \pi + \beta.
\end{align*}
Since $\alpha, \beta > 0$, we can add $2 \alpha + 2 C_f + \beta$ to the left-hand side, $2 \alpha + 2 \beta + 2 C_f$ to the  middle, and $2 \alpha + 2 \beta + 2 C_f + \frac{1}{2} \norm{x - z}^2 \pi$ to the right-hand side while preserving these inequalities. Let $\widetilde{\rho} := \sqrt{2(\alpha + \beta + C_f)}$, $c_1 := \frac{1}{2}$, and $C_1 := \frac{3}{2}$. Then, we have
\begin{align*}
c_1 \left(\norm{x - z}^2 \pi + \widetilde{\rho}^2 \right) \leq -\log \E p_N + \log N_f + \widetilde{\rho}^2 \leq C_1 \left(\norm{x - z}^2 \pi + \widetilde{\rho}^2 \right).
\end{align*}
\end{proof}

\begin{statement}[]
	\label{lem:thirdcondition}
With high probability, for $x \in \mathcal{M}_{1/\sqrt{2 \pi}}$ and $z$ the projection of $x$ onto $\mathcal{M}$, 
\begin{align*}
c_1 \left(\norm{x - z}^2 \pi + \rho^2 \right) \leq -\log p_N + \log N_f + \rho^2 \leq C_1 \left(\norm{x - z}^2 \pi + \rho^2 \right)
\end{align*}
for $0 < \rho < c$, with $c$ depending on $C_0, c_1, C_1, k, n$.
\end{statement}

\begin{proof}
From Lemmas \ref{claim2} and \ref{claim3}, w.h.p.\ for large enough $N$, $\inf_x p_N(x) \geq K_1 - \varepsilon_1$, or $\sup_x [-\log p_N(x)] \leq -\log (K_1 - \varepsilon_1)$. By uniform continuity,
\begin{align*}
\mathbb{P} \left[ \sup_{f \in \mathcal{F}} \abs{{-\log p_N(x)} + \log \E p_N(x)} \leq \varepsilon' \right] \geq 1 - \delta,
\end{align*}
where $\varepsilon' := \varepsilon_1 / \left(K_1 - \varepsilon_1\right)$. Thus, the statement of the lemma follows from Lemma \ref{claim6} with ${\rho} := \sqrt{2(\alpha + \beta + C_f + \varepsilon')}$, $c_1 := \frac{1}{2}$, and $C_1 := \frac{3}{2}$.
\end{proof}

We have proven all the conditions necessary in order to show that $-\log p_N(x) + \log N_f$ is an asdf. We summarize this in the next theorem, which is the major result of this section. We also prove that the constants we have defined are small enough to apply Theorem \ref{putativemanifold2} and state that $\Mput$ is a manifold with desirable properties.

\begin{thm}
	\label{mainkde}
$-\log p_N(x) + \log N_f$ is an approximate squared-distance function that meets the conditions in Theorem \ref{thm13fmn}. Consider the output set 
\begin{align*}
\mathcal{M}_{\mathrm{put}} = \set{z \in \mathcal{M}_{\min(c_3, c_4) \sigma /\sqrt{2 \pi}} \st \Pi_{hi}(z)\partial F(z) = 0}
\end{align*}
in the original coordinate system (i.e., the coordinates not scaled by $1/\sigma$, where $\sigma$ is the bandwidth of the KDE). By Theorem \ref{putativemanifold2}, $\Mput$ is a manifold whose reach is bounded below by $c \sigma$, where $c$ is a constant depending on $C_0, c_1, C_1, k, d,$ and $n$.  $\Mput$ converges to $\M$ in Hausdorff distance for increasing $N$; more specifically, $H(\M, \Mput) = O(\sigma^{5/4})$.
\end{thm}

\begin{proof}
Since $\log N_f$ is a constant, the first two conditions from Theorem \ref{thm13fmn} hold by Lemmas \ref{lem:firstcondition} and \ref{lem:secondcondition}. The third condition holds by Lemma \ref{lem:thirdcondition}. Thus, $-\log p_N(x) + \log N_f$ is an asdf.

For $\Mput$ to be a manifold, $\sigma \rho$ must be sufficiently small with respect to $\min(c_3, c_4)\sigma /\sqrt{2 \pi}$. We will show that for a small enough $\sigma$ and large enough $N$, $\sqrt{2(\alpha + \beta + C_f + \varepsilon')}$ (our choice of $\rho$) can be made as small as needed. Recall that $\tauhat = \tau/\sigma$ and $\taubarhat = \sigma^{-1/6}$. This implies
\begin{align*}
	C_f = \frac{d C^2 \sigma^{5/3}}{2 {\tau}^2} + \left(\frac{\sigma^{4/3}}{ {\tau}^2 } + \frac{2 \sqrt{2} \sqrt{\sigma}}{{\tau }}\right)\pi,
\end{align*}
which can be made as small as desired. $\alpha$ and $\beta$ can be bounded by using the fact that $-\log(1-p) < 2p$ if $p$ is sufficiently small. For a small enough $\sigma$, 
\begin{align*}
e^{dC^2 \sigma^{5/3}/(4\tau^2)} \leq e^{C_f},
\end{align*}
which gives the bound
\begin{align*}
\alpha & \leq 2 C_{neg}\\
&\leq \frac{2e^{-\sigma^{-1/3}  \pi/2}}{e^{dC^2 \sigma^{5/3}/(4\tau^2)} \sigma^d/(2V)}\\
&\leq 4V \sigma^{-d} e^{-\sigma^{-1/3}  \pi/4}.
\end{align*}
Similarly,
\begin{align*}
\beta &\leq 4e^{-\sigma^{-1/3}\pi/2}.
\end{align*}
As $\sigma$ tends to zero, so do these quantities. Finally, for large enough $N$, $\varepsilon_1$ is sufficiently small such that $\varepsilon' = \varepsilon_1 / \left(K_1 - \varepsilon_1\right)$ is as small as necessary. For a small enough $\sigma$ and a large enough $N$, $\rho = O(C_f^{1/2}) = O(\sigma^{1/4})$. This is sufficient to apply Theorem \ref{putativemanifold2}, which implies that $\Mput$ is a manifold with bounded reach that converges to $\M$ in Hausdorff distance. The Hausdorff distance $H(\M, \Mput)$ is $O(\sigma \rho)$, which is $O(\sigma^{5/4})$.
\end{proof}

\section{Local principal components analysis} \label{localpca}
A manifold $\M$ can be approximated by a finite collection of tangent spaces centered at a sufficiently dense set of points sampled from $\M$. \cite{fefferman2016testing} use this as motivation to define the concept of a cylinder packet; they also define a function $F^{\widebar{o}}$ and show that it is an asdf when coupled with a suitably constructed cylinder packet.  In this section we show that we can estimate tangent spaces directly from the data to create a cylinder packet; this leads to the construction of an approximate squared-distance function that satisfies Theorems \ref{thm13fmn} and \ref{putativemanifold2} and produces a putative manifold. 

\subsection{Definition of the asdf and selection of the bandwidth \boldmath{$\widebar{\tau}$}}
Let $C_p := \set{\mathrm{cyl}_i}$ be a collection of cylinders with centers $\set{x_i}$. Each cylinder is isometric to $\mathrm{cyl} := \widebar{\tau}(B_d \times B_{n - d})$. We choose $\widebar{\tau}$ so that it tends to zero but remains large compared to the distance between a sample point and its nearest neighbors. Since we are assuming uniform support on $\mathcal{M}$, for large $N$ we can choose $\widebar{\tau}$ on the order of $N^{-1/(d + \varepsilon)}$ for a small value of $\varepsilon$. 

Let $U_i$ be a proper rotation of $\mathrm{cyl}_i$, $Tr_i$ a translation, and $o_i$ a composition of a proper rotation and translation that moves the origin to $x_i$ and rotates the $d$-dimensional cross-section of $\mathrm{cyl}_i$ to $\R^d$. Define
\begin{align*}
	F^{\bar{o}}(z) := \frac{\sum_{\mathrm{cyl}_i \ni z}^{}\phi_{\mathrm{cyl}_i}\left(o_i^{-1}(z)\right) \theta\left(\Pi_d(o_i^{-1}(z))/(2 \widebar{\tau})\right)}{\sum_{\mathrm{cyl}_i \ni z}^{} \theta\left(\Pi_d(o_i^{-1}(z))/(2 \widebar{\tau})\right)},
\end{align*}
where $\mathrm{cyl}_i \in C_p$, $z \in \bigcup \mathrm{cyl}_i$, $\phi_{\mathrm{cyl}_i}(z)$ is the squared distance from $z$ to the $d$-dimensional cross-section of $\mathrm{cyl}_i$ and $\theta: \R^d \rightarrow [0, 1]$ is a bump function such that
\begin{enumerate}
\item{$\theta(y) = 0$ for $\norm{y} \notin (-1, 1)$}
\item{$\partial^{\alpha} \theta(y) = 0$ for $\abs{\alpha} \leq k$ and $y = 0$ or $\norm{y} \notin (-1, 1)$}
\item{$\abs{\partial^{\alpha} \theta(y)} < C$, a controlled constant for all $y$}
\item{$\theta(y) = 1$ for $\norm{y} < 1/4$}.
\end{enumerate}
Note that whether or not $F^{\bar{o}}(z)$ satisfies Theorem \ref{thm13fmn} depends on our choice of $C_p$; for convenience, we refer to the pair $\set{F^{\bar{o}}, C_p}$ as a putative asdf. $F^{\bar{o}}(z)$ measures the squared distances $\phi_{\mathrm{cyl}_i}$ to the central cross-section of each cylinder containing a given point $z$, and averages them using the bump function $\theta$. Let $\widehat{F}_z: B_n(0, 1) \rightarrow \R$ be a related function defined by $\widehat{F}_z(w) = F^{\bar{o}}(z + \widebar{\tau} \Theta(w))/\widebar{\tau}^2$, where $\Theta$ is an isometry that fixes the origin at $z$ and identifies the first $d$ coordinates with $T_z \M$. $\widehat{F}_z$ is essentially $F^{\bar{o}}$ analyzed in a coordinate system scaled up by $1/\widebar{\tau}$. This is analogous to our analysis of the kernel density estimator in the previous section, where we scaled the coordinate system by $1/\sigma$.

\subsection{Cylinder packets}
In order for $\set{F^{\bar{o}}, C_p}$ to be an asdf, $C_p$ needs to be a cylinder packet, which is a collection of cylinders that satisfies the geometric constraints given below in Definition \ref{cylinderpacket}. These conditions ensure that a cylinder packet doesn't contain pairs of cylinders that overlap too much or intersect at too great of an angle. This is motivated by our desire to estimate a manifold with bounded reach.

\begin{defn}[Cylinder packet]
\label{cylinderpacket}
Let $C_p$ be a collection of cylinders as above. $C_p$ is a cylinder packet if it satisfies the following conditions:
\begin{enumerate}
\item The number of cylinders is less than or equal to a constant factor times $\frac{V}{\tau^d}$.
\item Consider the set $S_i := \set{\mathrm{cyl}_{i_1}, \dots, \mathrm{cyl}_{i_{\abs{S_i}}}}$ of cylinders that intersect $\mathrm{cyl}_i$ and perform the rigid-body motion $o_i$. For each $\mathrm{cyl}_{i_j}$, there exists a translation $Tr_{i_j}$ and a proper rotation $U_{i_j}$ fixing $x_{i_j}$ so that
\begin{enumerate}
	\item For $1 \leq j \leq \abs{S_i}$, $Tr_{i_j}U_{i_j}\mathrm{cyl}_{i_j}$ is a translation of $\mathrm{cyl}_i$ by a vector with norm at least $\widebar{\tau}/3$.
	\item $\set{x_i} \bigcup \set{Tr_{i_j}U_{i_j}x_{i_j} \st 1 \leq j \leq \abs{S_i}} \bigcap \mathrm{cyl}_i$ forms a $\widebar{\tau}/2-$net of the $d$-dimensional cross-section of $\mathrm{cyl}_i$.
	\item For $1 \leq j \leq \abs{S_i}$ and $v \in \R^n$, $\norm{v - U_{i_j} v} < 2 \frac{\widebar{\tau}}{\tau} \norm{v - x_{i_j}}$.
	\item For $1 \leq j \leq \abs{S_i}$, $\norm{Tr_{i_j}(0)} < \frac{\widebar{\tau}^2}{\tau}$.
\end{enumerate}  
\end{enumerate}

\end{defn}

In Lemmas 16 and 17 due to \citet{fefferman2016testing}, it is shown that $\widehat{F}_z$ satisfies Theorem \ref{thm13fmn} when $C_p$ is a cylinder packet, meaning that $\set{F^{\bar{o}}, C_p}$ is an asdf. We include this towards the end of this section as Theorem \ref{fobarasdf1} and provide a sketch of the proof.

In the next lemma we construct a collection of cylinders $C_p^{Tan}$ whose central cross sections are derived from the tangent planes of the manifold and show that it is indeed a cylinder packet. The putative manifold actually has reach $c \widebar{\tau}$, so the right-hand sides of conditions 2(c) and (d) in Definition \ref{cylinderpacket} can be within a constant factor of what is given above.  

\begin{lem}
\label{idealcylpacket}
First, construct a set $\{x_i\}$ of centers. Assume the sample size is large enough to contain a $\widebar{\tau}/2-$net of $\mathcal{M}$ such that no two net points are within $\widebar{\tau}/2.9$ of each other. Let  $C_p^{Tan}$ be the collection of cylinders with centers $\{x_i\}$ and central cross sections contained in $T_{x_i} \mathcal{M}$. Then,  $C_p^{Tan}$ is a cylinder packet; we call it an ideal cylinder packet.
\end{lem}

\begin{proof}
We show that the conditions in Definition \ref{cylinderpacket} hold. Fix an $x_i$ and consider the set $S_i := \set{\mathrm{cyl}_{i_1}, \dots, \mathrm{cyl}_{i_{\abs{S_i}}}}$ of cylinders that intersect $\mathrm{cyl}_i$. Perform the rigid-body motion $o_i$ so that we are working in a convenient coordinate system. For each $x_{i_j} \in S_i$, define $U_{i_j}$ as a rotation fixing $x_{i_j}$ and rotating the central cross section of $\mathrm{cyl}_{i_j}$ so that it is parallel to $T_{x_i} \mathcal{M}$. Also, define $Tr_{i_j}$ as the translation that subsequently moves the central cross section so that it lies in $T_{x_i} \mathcal{M}$. Lemma \ref{lemmanifcovnum} implies the first condition.

Let $p$ be the projection of $x_{i_j}$ onto $T_{x} \mathcal{M}$. Federer's reach condition implies that
\begin{align*}
\norm{x_{i_j} - p} & \leq \frac{\norm{x_{i_j} - x_i}^2}{2 \tau}.
\end{align*}
Condition 2(d) holds since the right hand side must be less than $4 \widebar{\tau}^2/\tau$. Since $\norm{x_{i_j} - x_i} \geq \widebar{\tau}/2.9$, we also have
\begin{align*}
\norm{x_i - p} & \geq \sqrt{\norm{x_{i_j} - x_i}^2 - \frac{\norm{x_{i_j} - x_i}^4}{4 \tau^2}}\\
& \geq \sqrt{\frac{100 \widebar{\tau}^2}{841} - \frac{2500 \widebar{\tau}^4}{707281 \tau^2}}\\
& \approx \frac{10 \widebar{\tau}}{29} - \frac{125 \widebar{\tau}^2}{24389 \tau^2} + O(\widebar{\tau}^4),
\end{align*}
where the last line follows by a Taylor expansion. Since this can be made arbitrarily close to $\widebar{\tau}/2.9$,  $\norm{x_i - p} \geq \widebar{\tau}/3$ and Condition 2(a) is satisfied. Condition 2(b) follows from the fact that we started off with a $\widebar{\tau}/2-$net of the manifold; projecting $\{x_{i_1}, \dots, x_{i_{\abs{S_i}}}\}$ onto $T_{x_i} \mathcal{M}$ contracts interpoint distances so we end up with a $\widebar{\tau}/2-$net of the tangent space. 

To show the bound in 2(c), we need an expression for the angle between two nearby tangent spaces (in this case $T_{x_i} \mathcal{M}$ and $T_{x_{i_j}} \mathcal{M}$). In Lemma B.3 from a paper by \citet*{boissonnat2013constructing}, it is shown that the sine of the largest principal angle $\theta_1$ between $T_{x_i} \mathcal{M}$ and $T_{x_{i_j}} \mathcal{M}$ is less than or equal to $6 \norm{x_i - x_{i_j}}/\tau$, which is $12 \sqrt{2} \widebar{\tau}/\tau$ in our setup. Now, translate the origin to $x_{i_j}$ and translate $T_{x_i} \mathcal{M}$ so that it contains $x_{i_j}$. Without loss of generality, let $v \in T_{x_{i_j}} \mathcal{M}$. Let $\set{e_i}_1^d$ and $\set{{\widehat{e}_i}}_1^d$ be orthonormal bases for $T_{x_{i_j}}$ and $T_{x_{i}}$, respectively, so that the angle between $e_i$ and $\widehat{e}_i$ is the principal angle $\theta_i$. Define $U_{i_j}$ as the rotation that maps $\set{e_i}_1^d$ onto $\set{{\widehat{e}_i}}_1^d$. Let $\set{v_i}_1^d$ be the components of $v$ and $U_{i_j} v$ with respect to the appropriate bases. Then we have the following:
\begin{align*}
\norm{v - U_{i_j} v} & \leq \norm{\sum_{i = 1}^d v_i e_i - \sum_{i = 1}^d v_i \widehat{e}_i}\\
& \leq \sum_{i = 1}^d \norm{v_i e_i - v_i \widehat{e}_i}\\
& \leq \sum_{i = 1}^d \abs{v_i }\norm{e_i - \widehat{e}_i}\\
&\leq \norm{e_1 - \widehat{e}_1}\norm{v}_{1} \\
&\leq \norm{e_1 - \widehat{e}_1} \sqrt{d} \norm{v}.
\end{align*}
Using the law of cosines, $\norm{e_i - \widehat{e}_i} \leq \sqrt{2 - 2 \cos \theta_i}$. From the bound on $\sin \theta_1$ and a Taylor expansion of $\cos \arcsin \left(12\sqrt{2} \widebar{\tau}/\tau\right)$, we can show
\begin{align*}
2 - 2 \cos \theta_1 & \leq \frac{288 \widebar{\tau}^2}{\tau^2} + O(\widebar{\tau}^4)\\
& \leq \frac{576 \widebar{\tau}^2}{\tau^2}
\end{align*}
for large enough $N$. Thus,
\begin{align*}
\norm{v - U_{i_j} v} & \leq \frac{24 \sqrt {d} \widebar{\tau}}{\tau} \norm{v},
\end{align*}
which shows 2(c).
\end{proof}

\begin{cor}
\label{admissiblecylpack}
First, construct a set $\{x_i\}$ of centers. Assume the sample size is large enough to contain a $\widebar{\tau}/2-$net of $\mathcal{M}$ such that no two net points are within $\widebar{\tau}/2.9$ of each other. The collection of cylinders with centers $\{x_i\}$ and central cross sections within $O(\widebar{\tau}/\tau)$ of $T_{x_i} \mathcal{M}$ in operator norm is a cylinder packet; we call it an admissible cylinder packet.
\end{cor}

\subsection{Constructing an admissible cylinder packet with local PCA}
Usually we only have access to points sampled from the manifold and not their associated tangent spaces. It is easy to see that we can also construct a cylinder packet if we can estimate the tangent spaces accurately enough (as stated in Corollary \ref{admissiblecylpack}). Let $C_p^{\widehat{Tan}}$ be a collection of cylinders constructed by using the same net points as in Lemma \ref{idealcylpacket} and performing local PCA to estimate the $d$-dimensional cross-sections.  In this section, we show that $C_p^{\widehat{Tan}}$ is an admissible cylinder packet. The $d$-dimensional cross-sections are estimated as follows. Given a sample point $z \in \mathcal{M}$, we construct the PCA matrix $N_z^{-1} Y Y^\top$, where $Y$ has columns consisting of the $N_z$ sample points lying within $\widebar{\tau}(B_d \times B_{n - d})$. (We are using a coordinate system centered at $z$ whose first $d$ coordinates lie in $T_z \mathcal{M}$). Using the eigenvectors of $N_z^{-1} Y Y^\top$, we can get an estimate $\widehat{T_z \mathcal{M}}$ of the tangent space at $z$. We show that this is close to the true tangent space by using the Davis-Kahan $\sin \theta$ theorem. The version stated below is due to \citet*{yu2015useful}.

Let $\|\cdot\|_F$ denote the Frobenius norm of a matrix. Suppose $V, \widehat{V} \in \R^{n \times d}$ both have orthonormal columns. Theorem \ref{thm:daviskahan} gives an upper bound on $\norm{\sin \theta(\widehat{V}, V)}_F$, where $\theta(\widehat{V}, V)$ is the $d \times d$ diagonal matrix whose diagonal consists of the principal angles between the column spaces of $V$ and $\widehat{V}$ and $\sin \theta(\widehat{V}, V)$ is defined entrywise.  The principal angles are given by $\set{\cos^{-1} \zeta_1, \dots, \cos^{-1} \zeta_d}$, where $\set{\zeta_1, \dots, \zeta_d}$ are the singular values of
$\widehat{V}^TV$.

\begin{thm}[Davis-Kahan $\sin \theta$ Theorem]
	\label{thm:daviskahan} 
Let $\Lambda, \widehat{\Lambda} \in \R^{n\times n}$ be symmetric, with eigenvalues $\lambda_1 \geq \dots \geq \lambda_n$ and $\widehat{\lambda}_1 \geq \dots \geq \widehat{\lambda}_n$ respectively. Let $1 \leq d \leq n$ and assume $\lambda_{d} - \lambda_{d + 1} > 0$. Let $V = (v_1, \dots, v_d) \in \R^{n \times d}$ and $\widehat{V} = (\widehat{v}_1, \dots, \widehat{v}_d) \in \R^{n\times d}$ have orthonormal columns satisfying $\Lambda v_j = \lambda_j v_j$ and $\widehat{\Lambda} \widehat{v}_j = \widehat{\lambda}_j \widehat{v}_j$ for $j = 1, \dots, d$. Then
\begin{align*}
\norm{\sin \theta(\widehat{V}, V)}_F \leq \frac{2 \norm{\widehat{\Lambda} - \Lambda}_F}{\lambda_d - \lambda_{d + 1}}.
\end{align*}
\end{thm}

We also make use of the following concentration inequality due to \citet*{ahlswede2002strong}. Let $A \preceq B$ mean that $A - B$ is positive semidefinite.
 \begin{thm}
 \label{AW}
Let $a_1, \dots, a_k$ be i.i.d. random positive semidefinite $d \times d$ matrices with expected value $\E[a_i] = M \succeq \mu I$ and $a_i \preceq I$. Then for all $\epsilon \in [0, 1/2]$,
\begin{align*} \mathbb{P}\left[\frac{1}{k} \sum_{i=1}^k a_i \not\in [(1-\epsilon)M, (1 + \epsilon)M]\right] \leq 2D \exp\left\{\frac{-\epsilon^2\mu k}{2\ln 2}\right\}.
\end{align*}
\end{thm}

We now prove the key result of this section.

\begin{thm} 
\label{mainpca}
Let $z$ be a point sampled from $\mathcal{M}$. Translate $z$ to the origin, and let the first $d$ coordinates lie in $T_z \mathcal{M}$. Let $Y$ be a matrix whose columns consist of the $N_z$ sample points $\{y_i\}$ lying within $\widebar{\tau}(B_d \times B_{n - d})$. Let $X$ be a matrix whose columns are the projections $\{x_i\}$ of $\{y_i\}$ onto $T_z \mathcal{M}$. Construct the matrices $N_z^{-1}X X^\top$ and $N_z^{-1}Y Y^\top$, and let $V$ and $\widehat{V}$ be their respective matrices of eigenvectors. Then, w.h.p.,

\begin{align*}
 	\norm{\sin \theta(\widehat{V}, V)}_F & \leq \frac{\left(\frac{2\widebar{\tau}^3}{\tau} + \frac{2 \widebar{\tau}^4}{\tau^2}\right) (d + 2)}{(1 - \epsilon)\widebar{\tau}^{2} \left(1 + \frac{C^2 \widebar{\tau}^2}{\tau^2}\right)^{-d/2}},
\end{align*}
where $\epsilon \in [0, 1/2]$.
\end{thm}

\begin{proof}
Clearly, $N_z$ increases with $N$. We can assume the matrices in the statement of the theorem can be defined. We apply Theorem \ref{thm:daviskahan} with $N_z^{-1}X X^\top$ and $N_z^{-1}Y Y^\top$ corresponding to $\Lambda$ and $\widehat{\Lambda}$, respectively. 

We start off by bounding the numerator $\left\|N_z^{-1} \left(YY^{\top} - XX^{\top}\right)\right\|_F$. This is easiest if we consider $Y$ as a perturbation of $X$ by the matrix $P$ since we can control $P$ using Federer's reach condition. This gives:
\begin{align*}
	Y Y^\top & = (X + P)(X + P)^\top \\
	& = X X^\top + X P^\top + P X^\top + P P^\top.
\end{align*}
Therefore,
\begin{align*}
	\left\|N_z^{-1} \left(YY^{\top} - XX^{\top}\right)\right\|_F & \leq N_z^{-1}\left\|X P^\top + P X^\top + P P^\top\right\|_F\\
	& \leq N_z^{-1}\left(\left\|X P^\top\right\|_F + \left\|P X^\top\right\|_F + \left\|P P^\top\right\|_F\right)\\
	& \leq N_z^{-1}\left(\left\|X\right\|_F\left\|P^\top\right\|_F + \left\|P\right\|_F\left\|X^\top\right\|_F+ \left\|P\right\|_F\left\|P^\top\right\|_F\right).
\end{align*} 
Because each column of X has norm less than or equal to $\widebar{\tau}$,   $\norm{X} \leq  \sqrt{N_z}\widebar{\tau}$. By Federer's reach condition, we have
 \begin{align*}
 |y_i - x_i| & \leq  \frac{|z - y_i|^2}{2\tau}\\
                     & \leq \frac{\left(\sqrt{2} \widebar{\tau}\right)^2}{2 \tau}, 
  \end{align*}
  which implies that $\norm{P} \leq \sqrt{N_z {\widebar{\tau}^4}/({\tau^2})}$. Thus,
 \begin{align*}
	\left\|N_z^{-1} \left(YY^{\top} - XX^{\top}\right)\right\|_F & \leq \frac{\widebar{\tau}^3}{\tau} + \frac{\widebar{\tau}^4}{\tau^2}.
 \end{align*}

Now we need to bound $\lambda_d - \lambda_{d + 1}$. Let $\lambda_1 \geq \dots \geq \lambda_n$ be the eigenvalues of $N_z^{-1}X X^\top$, and let $\mu_1 \geq \dots \mu_n$ be the eigenvalues of $M := \E\left[N_z^{-1}X X^\top\right]$. We see that $\lambda_{d+1} = \dots = \lambda_{n} = \mu_{d+1} = \dots = \mu_{n} = 0$. So, we only need a lower bound for $\lambda_d$, which we can get by relating its value to $\mu_d$ through a concentration inequality. Assuming the first $d$ coordinates are aligned with the eigenvectors of $M$, $\mu_d$ is the variance in the direction $x_d$. $M$ is the population covariance matrix of the probability measure $\mathcal{P}$ on $T_z \mathcal{M} \cap B_d\left(\widebar{\tau}\right)$ that is the pushforward of the uniform measure on $\mathcal{M} \cap \widebar{\tau}(B_d \times B_{n - d})$. From Lemma \ref{lemmapushforward}, we know $\mathcal{P}$ has a density $p(x)$ that is greater than or equal to $\left(1 + {C^2 \widebar{\tau}^2}/{\tau^2}\right)^{-d/2}$ multiplied by a normalizing factor, which in this case is just $\mathrm{Vol}(B_d(\widebar{\tau}))^{-1}$. 

We have the following bound for $\mu_d$:
\begin{align*}
 \mu_d & = \int_{B_d(\widebar{\tau})} x_d^2 d \mathcal{P}(x)\\
 & \geq \int_{B_d(\widebar{\tau})} x_d^2 \frac{\left(1 + \frac{C^2 \widebar{\tau}^2}{\tau^2}\right)^{-d/2}}{\mathrm{Vol}(B_d(\widebar{\tau}))} d \mathcal{L}_d(x)\\
 & \geq  \int_{0}^{{\widebar{\tau}}}\int_{0}^{\pi}\dots\int_{0}^{\pi}\int_{0}^{2 \pi} \left(r \sin{\phi_{d - 1}} \prod_{j = 1}^{d - 2}\sin{\phi_j}\right)^2 \frac{\left(1 + \frac{C^2 \widebar{\tau}^2}{\tau^2}\right)^{-d/2}}{\mathrm{Vol}(B_d(\widebar{\tau}))} dV\\
 & =  \frac{\Gamma(d/2 + 1) \left(1 + \frac{C^2 \widebar{\tau}^2}{\tau^2}\right)^{-d/2}}{\pi^{d/2} \widebar{\tau}^d} \int_{0}^{{\widebar{\tau}}}\int_{0}^{\pi}\dots\int_{0}^{\pi}\int_{0}^{2 \pi} r^{d+1} \prod_{j = 1}^{d - 1} \sin^{d - j + 1}{\phi_j} dr \prod_{j = 1}^{d - 1} d\phi_{j}.
\end{align*}
The third line follows by a change of coordinates. Substitute
\begin{align*}
\set{x_1 \mapsto r \cos{\phi_1}, x_{2 \leq i \leq d - 1} \mapsto r \cos{\phi_i} \prod_{j = 1}^{i - 1}\sin{\phi_j}, x_d \mapsto r \sin{\phi_{d - 1}} \prod_{j = 1}^{d - 2}\sin{\phi_j}},
\end{align*} 
and let
\begin{align*}
dV := r^{d - 1} \prod_{j = 1}^{d - 2} \sin^{d - j - 1} {\phi_j} dr d\phi_1 \dots d\phi_{d - 1}.
\end{align*}
The integral in the fourth line can be evaluated by noting that $\displaystyle \int_{0}^{\widebar{\tau}} r^{d + 1} dr = {\widebar{\tau}^{d + 2}}\bigg/{(d + 2)}$, $\displaystyle \int_{0}^{2 \pi} \sin^2 \phi_{d - 1} d\phi_{d - 1} = \pi$, and $\displaystyle \int_{0}^{\pi} \sin^{d - j + 1} \phi_{j}  d\phi_{j} = \frac{\sqrt{\pi} \Gamma\left((d - j + 2)/2\right)}{\Gamma\left(1 + (d - j + 1)/2\right)}$ for $1 \leq j \leq d - 2$. We can simplify (by telescoping)
\begin{align*}
\prod_{j = 1}^{d - 2} \frac{\sqrt{\pi} \Gamma\left((d - j + 2)/2\right)}{\Gamma\left(1 + (d - j + 1)/2\right)} = \frac{\pi^{(d - 2)/2}}{\Gamma(1 + d/2)}.
\end{align*}
Therefore,
\begin{align*}
  \mu_d & \geq  \frac{1}{d + 2} \widebar{\tau}^{2} \left(1 + \frac{C^2 \widebar{\tau}^2}{\tau^2}\right)^{-d/2}.
\end{align*}

$N_z^{-1}X X^\top$ and $M$ are zero outside the upper left $d \times d$ block.  Call their nonzero blocks $\Xi$ and $\widetilde{\Xi}$, respectively; clearly these matrices have eigenvalues $\{\lambda_i\}_{1}^{d}$ and $\{\mu_i\}_{1}^{d}$. $\Xi$ can be written as the empirical average ${N_z}^{-1}\sum_{1 = 1}^{N_z} x_{i, d}x_{i, d}^{\top}$, where the $\{x_{i, d}\}$ are the first $d$ coordinates of the $\{x_i\}$. Note that $x_{i, d}x_{i, d}^{\top} \preceq \left\|x_{i, d}x_{i, d}^{\top}\right\|_F I \preceq \widebar{\tau}^2 I$. For $N_z$ large enough, this implies $x_{i, d}x_{i, d}^{\top} \preceq I$. Additionally, since $\mu_d > 0$ is the smallest eigenvalue of $\widetilde{\Xi}$, we have $\widetilde{\Xi} \succeq \mu_d I$. This is sufficient to apply Theorem \ref{AW}. So, for all $\epsilon \in [0, 1/2]$,
\begin{align*}
\mathbb{P}\left[\Xi \notin \left[(1 - \epsilon)\widetilde{\Xi}, (1 + \epsilon)\widetilde{\Xi}\right]\right] \leq 2d \exp\left\{-\frac{\epsilon^2 \mu_d N_z}{2 \log 2}\right\}.
\end{align*}
The matrix interval is in terms of the positive semidefinite ordering, so $\Xi \succeq (1 - \epsilon) \widetilde{\Xi}$ w.h.p. This implies $\lambda_d \geq (1 - \epsilon) \mu_d$.

Now, applying Theorem \ref{thm:daviskahan},
 \begin{align*}
 	\norm{\sin \theta(\widehat{V}, V)}_F &\leq \frac{2\left\|N_z^{-1} \left(YY^{\top} - XX^{\top}\right)\right\|_F}{\lambda_d}\\
 	& \leq \frac{\left(\frac{2\widebar{\tau}^3}{\tau} + \frac{2 \widebar{\tau}^4}{\tau^2}\right) (d + 2) }{(1 - \epsilon)\widebar{\tau}^{2} \left(1 + \frac{C^2 \widebar{\tau}^2}{\tau^2}\right)^{-d/2}}\\
 	& = O\left(\frac{\widebar{\tau}}{\tau}\right).
  \end{align*}
\end{proof}

\subsection{\boldmath{$\set{F^{\bar{o}}(z), C_p^{\widehat{Tan}}}$} is an asdf}

In Theorem \ref{fobarasdf1}, we sketch a proof that $\set{F^{\bar{o}}, C_p}$ is an asdf for an arbitrary cylinder packet $C_p$; we also show that Theorem \ref{putativemanifold2} applies. Since we showed that $C_p^{\widehat{Tan}}$ is an admissible cylinder packet, it follows immediately that $\set{F^{\bar{o}}(z), C_p^{\widehat{Tan}}}$ is an asdf.

\begin{thm}
\label{fobarasdf1}
Assume that we are given a cylinder packet $C_p$. $\set{F^{\bar{o}}, C_p}$ is an approximate squared-distance function that meets the conditions in Theorem \ref{thm13fmn}. Furthermore, by Theorem \ref{putativemanifold2}, the output set (in the original coordinate system)
\begin{align*}
\mathcal{M}_{\mathrm{put}} = \set{z \in \mathcal{M}_{\min(c_3, c_4)\widebar{\tau}} \st \Pi_{hi}(z)\partial F(z) = 0}.
\end{align*}
is a manifold whose reach is bounded below by $c \widebar{\tau}$, where $c$ is a constant depending on $C_0, c_1, C_1, k, d,$ and $n$.  $\Mput$ converges to $\M$ in Hausdorff distance for increasing $N$; more specifically, $H(\M, \Mput) = O(\widebar{\tau}^2)$.
\end{thm}

\begin{proof}
$\widehat{F}_z$ is $C^k$-smooth by the chain rule and the smoothness of the projection, distance, and bump functions. $\partial^\alpha (\widehat{F}_z(w)) \leq C_0$ for $w \in B_n(0,1)$, $|\alpha| \leq k$, and $C_0$ depending only on $n$ and $k$. This is true by the chain rule since the bounds on the derivatives of the bump function and the distance function can be directly calculated. After rescaling by $\widebar{\tau}$, these depend only on $n$ and $k$.

The third condition is satisfied by setting $\rho$ equal to $c_{\rho}\widebar{\tau}/\tau$, where $c_{\rho}$ is a constant depending on the geometry of $C_p$. Let $z' := z + \widebar{\tau} \Theta(w)$. If $C_p$ is a cylinder packet, the distance from $\Pi_\M z'$ to the central cross-section of any cylinder containing $z'$ is on the order of $\widebar{\tau}^2/\tau$. $\widehat{F}_z(w)$ is a rescaled convex combination of the squared distance between $z'$ and the central cross section of the cylinders containing it. That is, $\widehat{F}_z(w)$ is essentially $\widebar{\tau}^{-2} \sum b_i \left(\norm{\Pi_\M z' - z'} + c_i \widebar{\tau}^2/\tau\right)^2$, where $\sum b_i = 1$ and the $c_i$ depend on $C_p$. Thus, setting $\rho^2$ to $c_{\rho}^2\widebar{\tau}^2/\tau^2$ satisfies the third condition of Theorem \ref{thm13fmn} for appropriate values of $c_1$ and $C_1$:
\begin{align*}
c_1\left(\frac{\norm{\Pi_\M z' - z'}^2}{\widebar{\tau}^2} + {\rho^2}\right) \leq \widehat{F}_z(w) + {\rho^2} \leq C_1\left(\frac{\norm{\Pi_\M z' - z'}^2}{\widebar{\tau}^2}  + {\rho^2}\right),
\end{align*}   
where $0 < \rho < c$, with $c$ depending on $C_0, c_1, C_1, k, n$.

For Theorem \ref{putativemanifold2} to apply, $\widebar{\tau} \rho$ must be sufficiently small with respect to $\min(c_3, c_4)\widebar{\tau}$. This is clearly true because $\widebar{\tau} \rho/(\min(c_3, c_4)\widebar{\tau}) = O(\widebar{\tau}/\tau)$, which can be made as small as desired. The Hausdorff distance $H(\M, \Mput)$ is $O( \widebar{\tau} \rho)$, which is $O(\widebar{\tau}^2)$.
\end{proof}

\begin{thm}
$\set{F^{\bar{o}}(z), C_p^{\widehat{Tan}}}$, where $C_p^{\widehat{Tan}}$ is a cylinder packet constructed using local PCA is an approximate squared-distance function that meets the conditions in Theorem \ref{thm13fmn}. Furthermore, by Theorem \ref{putativemanifold2}, the output set 
\begin{align*}
\mathcal{M}_{\mathrm{put}} = \set{z \in \mathcal{M}_{\min(c_3, c_4) \widebar{\tau}} \st \Pi_{hi}(z)\partial F(z) = 0}
\end{align*}
is a manifold whose reach is bounded below by $c \widebar{\tau}$, where $c$ is a constant depending on $C_0, c_1, C_1, k, d,$ and $n$. $\Mput$ converges to $\M$ in Hausdorff distance for increasing $N$; more specifically, $H(\M, \Mput) = O(\widebar{\tau}^2)$.
\end{thm}

\begin{proof}
This is a direct consequence of Theorem \ref{fobarasdf1}, Corollary \ref{admissiblecylpack}, and Theorem \ref{mainpca}.
\end{proof}

\section{Simulations} \label{simulations}

In this section, we present simulation results showing that the two asdfs considered in this paper can be used to find a discretized version of a putative manifold. All simulations were performed using the following gradient descent algorithm based on subspace-constrained mean shift \citep*{ozertem2011locally}. 
\begin{enumerate}
\item Initialize a mesh of points on which to perform gradient descent. They can be sample points with or without added noise.
\item Perform the following for each mesh point $x$:
	\begin{enumerate}
	\item Calculate the gradient $g$ and the Hessian $H$ of the asdf $f$.
	\item Let $V$ be a matrix whose columns are the eigenvectors corresponding to the largest $n - d$ eigenvalues of $H$.
	\item Calculate $V V^\top g$ and take a step in this direction.
	\item Go to step (a) until a tolerance condition is met.
	\end{enumerate}
\end{enumerate}

We applied this algorithm to data points sampled from three different manifolds contained in the unit ball of a Euclidean space: a circle embedded in $\R^2$, a closed curve embedded in $\R^3$, and a sphere embedded in $\R^3$. We sampled 1000 points from each manifold and used this data to construct asdfs based on KDE and local PCA. We then sampled 1000 additional points and added Gaussian noise with a standard deviation of 0.05; these were used as the starting mesh points. Finally, we ran the algorithm and took the final output to be points lying on the putative manifold. Figure \ref{figure1} shows an example of each of the three manifolds for each asdf. To get a sense of the accuracy of this procedure, we found the RMS distance of each putative manifold to a 10000 point sample (i.e., an approximate net) derived from the original manifolds. The average RMS distance from 100 trials is given in Table \ref{table1}. 

\begin{table}[ht]
\centering

\begin{tabular}{rrrr}
  \hline
 & Circle $\subset \R^2$ & Curve $\subset \R^3$ & Sphere $\subset \R^3$ \\ 
  \hline
KDE & 0.000433 & 0.000990 & 0.00221 \\ 
  Local PCA & 0.000146 & 0.000453 & 0.000603 \\ 
   \hline
\end{tabular}
\caption{Average RMS distance for subspace-constrained gradient descent on two asdfs} \label{table1}
\end{table}

\begin{figure}[htp]

  \subcaptionbox*{\label{fig1:a}}{\includegraphics[width=.4\linewidth]{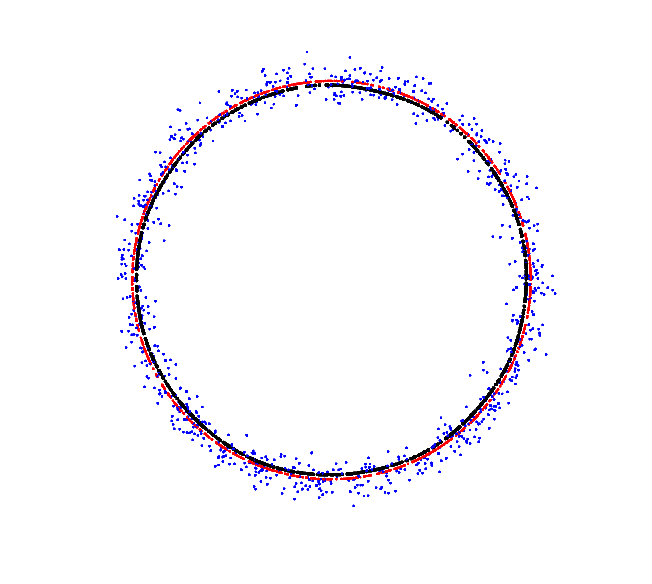}}\hfill%
  \subcaptionbox*{\label{fig1:b}}{\includegraphics[width=.4\linewidth]{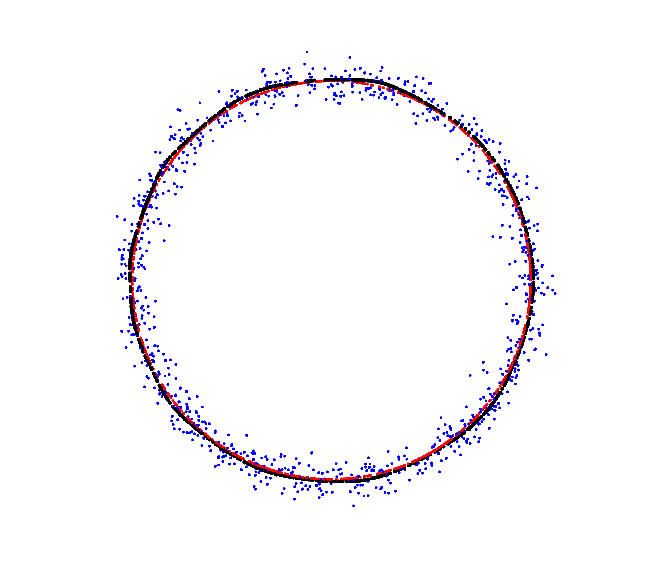}}%

 \vspace{-2em}

  \subcaptionbox*{\label{fig1:c}}{\includegraphics[width=.4\linewidth]{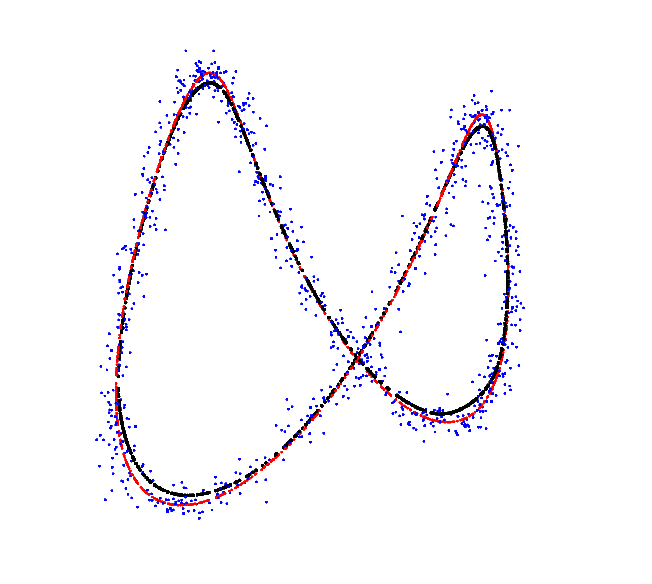}}\hfill%
  \subcaptionbox*{\label{fig1:d}}{\includegraphics[width=.4\linewidth]{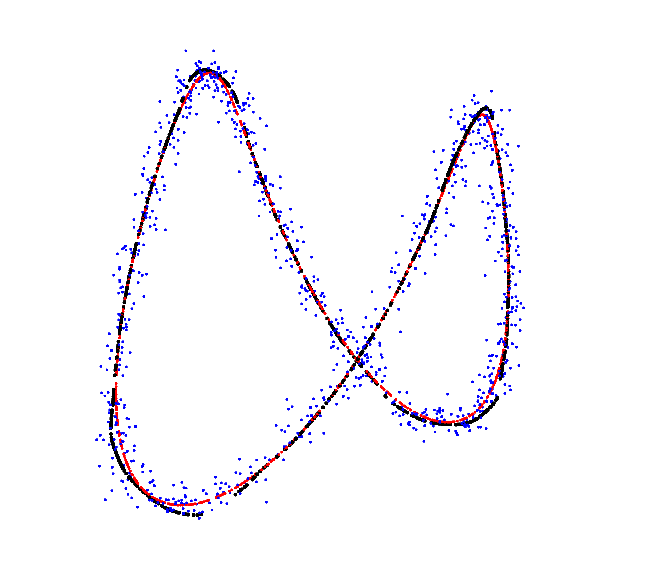}}%

 \vspace{-2em}

  \subcaptionbox*{\label{fig1:e}}{\includegraphics[width=.4\linewidth]{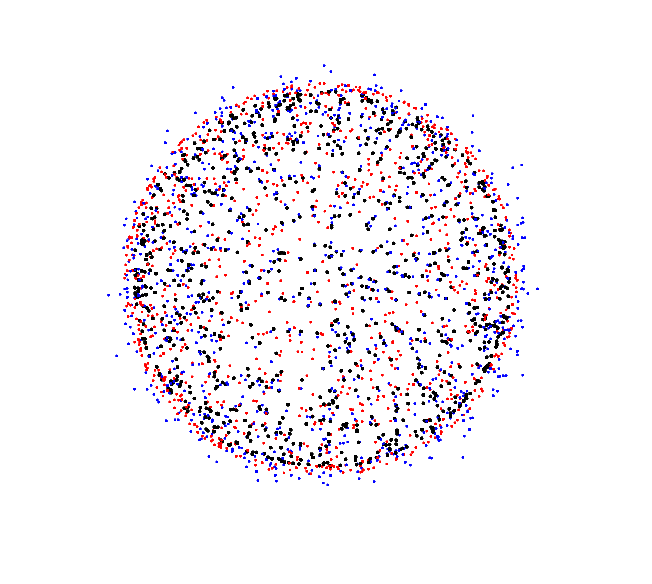}}\hfill%
  \subcaptionbox*{\label{fig1:f}}{\includegraphics[width=.4\linewidth]{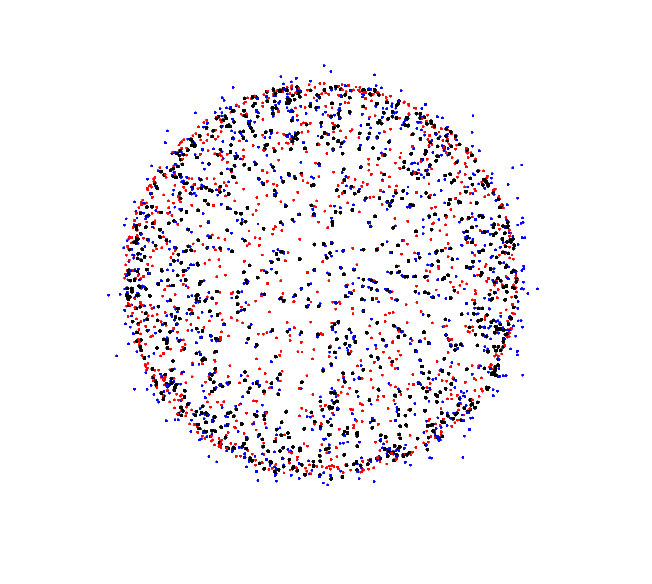}}%
        \captionsetup{format=hang}
 \caption{Circle $\subset \R^2$, curve $\subset \R^3$, and sphere $\subset \R^3$ examples (top to bottom). Red points
 are sampled from the manifold, blue points are sampled points with added Gaussian noise (sd = 0.05), and
 black points are the output points obtained by subspace-constrained gradient descent. The left column uses an asdf based on KDE, and the right column uses an asdf based on local PCA.} \label{figure1}
\end{figure}

\section{Discussion} \label{discussion}

In this paper, we showed that if we are provided with data sampled from a manifold $\mathcal{M}$, we can use two different asdfs to construct an estimator of $\mathcal{M}$. The asdfs are based on kernel density estimation and local PCA, which are conceptually easy to understand and mainstays of nonparametric estimation. The estimator is a manifold itself, and there are concrete bounds on its geometry (for example, its reach). These bounds are derived from an application of the implicit function theorem and are given in a key theorem of \citet*{fefferman2016testing}. Our contribution in this paper is to create asdfs that can be calculated directly from the data as well as to give bounds on the reach and Hausdorff distance that depend on the sample size and properties of the asdfs. In the future, we aim to work on several natural extensions of our results. It remains to be seen what can be said about an estimator derived from a sample contaminated with noise (potentially bounded or sub-Gaussian). Additionally, it would be of theoretical interest to see how precise we can make the constants in this paper, including the ones derived from \cite{fefferman2016testing}.
 
\section*{Acknowledgments}

We are grateful to Charlie Fefferman and Sanjoy Mitter for several illuminating discussions on local PCA. We would also like to thank Charlie for discussing the main results from \cite{fefferman2016testing} with us, especially Theorem 13 and Lemma 14. We are also grateful to Marina Meila, Johannes Lederer, Emo Todorov, and Yen-Chi Chen. Their comments on a preliminary draft of this paper were very helpful. We would further like to thank Yen-Chi for advice that we used in constructing the simulation studies used in this paper. Finally, both authors were partially supported by NSF grant DMS 1620102.

\bibliography{papernewbib}

\end{document}